\documentclass[a4paper,11pt]{article}

\usepackage[italian,english]{babel}
\usepackage{latexsym}
\usepackage{amsfonts}
\usepackage{amssymb}
\usepackage{amsthm}
\usepackage{amsmath}
\usepackage{verbatim}
\usepackage{amscd}
\usepackage{graphicx}
\usepackage{tikz-cd}
\usepackage{tikz}
\usepackage{epstopdf}
\usepackage[font=small,labelfont=bf]{caption}
\usetikzlibrary{matrix,arrows,decorations.pathmorphing}

\title{A gluing formula for Reidemeister-Turaev torsion}
\author{Stefano Borghini}
\date{}

\newtheorem*{teo*}{Theorem}
\newtheorem{teo}{Theorem}[section]
\newtheorem{prp}[teo]{Proposition}

\theoremstyle{definition}

\newtheorem*{ntz}{Notation}
\newtheorem{lmm}[teo]{Lemma}
\newtheorem*{ack}{Acknowledgements}

\theoremstyle{remark}
\newtheorem{nta}[teo]{Remark}

\addto\captionsenglish{%
  \renewcommand{\figurename}{Fig.}
}

\begin{document}

\maketitle

\begin{abstract}
We extend Turaev's theory of Euler structures and torsion invariants on a 3-manifold $M$ to the case of vector fields having generic behavior on $\partial M$. 
This allows to easily define gluings of Euler structures and to develop a completely general gluing formula for Reidemeister torsion of 3-manifolds.
Lastly, we describe a combinatorial presentation of Euler structures via stream-spines, as a tool to effectively compute torsion.
\end{abstract}

\section*{Introduction}
Reidemeister torsion is a classical topological invariant introduced by Reidemeister (\cite{reidemeister}) in order to classify lens spaces.  
Significant improvements in the study of this invariant have been made by Milnor (\cite{milnor}), who discovered connections between torsion and Alexander polynomial, and Turaev (\cite {turaev}), who showed that the ambiguity in the definition of Reidemeister torsion could be fixed by means of Euler structures (i.e., equivalence classes of non-singular vector fields).  Actually, to completely fix the ambiguity, Turaev introduced the additional notion of homology orientation, but we will not consider it (see Remark~\ref{homology orientation}).

Recently, Reidemeister torsion has proven its utility in a number of topics in 3-dimensional topology.
For instance, Reidemeister torsion is the main tool in the definition of the Casson-Walker-Lescop invariants (\cite{lescop}) and of Turaev's maximal abelian torsion (\cite{turaev1}), which in turn has been proved to be equivalent (up to sign) to the Seiberg-Witten invariants on 3-manifolds (if the first Betti number is $\neq 0$).

The aim of this paper is to describe the behavior of Reidemeister torsion on 3-manifolds with respect to gluings along a surface.
This certainly is a problem of interest: to name a few examples of the importance of gluings, Heegaard splittings are one of the main ingredients in the construction of Heegaard Floer homology (\cite{szabo}), and multiplicativity with respect to gluings is one of the fundamental axioms of Topological Quantum Field Theories (\cite{atiyah}).

Our reference model is the following. We consider a (closed) 3-manifold endowed with an Euler structure, 
and we split it into two submanifolds $M_1,M_2$ along a surface $S$. As we have complete freedom in the choice of $S$, we need to define Euler structures of $M_1,M_2$ as equivalence classes of vector fields with a generic behavior on the boundary $S$.

The definition of Euler structure is the object of Section~\ref{sctn1}. In particular, we describe the action of the first integer homology group on combinatorial and smooth Euler structures and we recover Turaev's reconstruction map $\Psi$, i.e., an equivariant bijection from combinatorial to smooth Euler structure.

In Section~\ref{sctn2} we define Reidemeister torsion of a pair $(M,\mathfrak{e}^c)$, where $M$ is a 3-manifold and $\mathfrak{e}^c$ is a combinatorial Euler structure. If $\mathfrak{e}^s$ is a smooth Euler structure, the torsion of $(M,\mathfrak{e}^s)$ is defined as the torsion of $(M,\Psi^{-1}(\mathfrak{e}^s))$. We emphasize that we need a way to explicitly invert the reconstruction map in order to effectively compute torsion (this will be the subject of Section \ref{sctn4}).

Section~\ref{sctn3} is devoted to the proof of Theorem~\ref{thm.gluing}, informally stated below:

\begin{teo*}
Reidemeister torsion acts multiplicatively with respect to gluings. Namely, given a smooth compact oriented closed 3-manifold $M$ and an embedded surface $S$ splitting $M$ into two submanifolds $M_1,M_2$:
\begin{itemize}
\item a representative of an Euler structure $\mathfrak{e}$ on $M$ induces Euler structures $\mathfrak{e}_1,\mathfrak{e}_2$ on $M_1,M_2$;
\item Reidemeister torsion of $(M,\mathfrak{e})$ is the product of Reidemeister torsions of $(M_1,\mathfrak{e}_1)$ and $(M_2,\mathfrak{e}_2)$, times a corrective term $\mathfrak{T}$ coming from the homologies.
 \end{itemize}
\end{teo*}
This theorem greatly extends a preceiding result due to Turaev (\cite[Lemma~VI.3.2]{turaev1}), which holds in his very special setting only ($S$ is an union of tori and Euler structures are equivalence classes of vector fields everywhere transversal to the boundary). The closure of $M$ is not a necessary hypothesis, we have assumed it only to simplify notations and proof (an extension to the case with boundary is stated, without proof, in Remark~\ref{rmk.gluing}). 
In the end of Section~\ref{sctn3} we show some computations, aimed at simplify the term $\mathfrak{T}$.

Finally, in Section~\ref{sctn4} we describe a combinatorial encoding of Euler structures in order to explicitly invert the reconstruction map $\Psi$.
The key tool will be a generalized version of standard spines (the stream-spines described in \cite{petronio}), that allows to encode vector fields with generic behavior on the boundary.

In our work, we have focused on the abelian version of Reidemeister torsion (in order to simplify the algebraic machinery); all the results extend with minimal modifications to the non-abelian case. Section~\ref{sctn1}, \ref{sctn2}, \ref{sctn4} follow the exposition and the ideas of 
\cite{benedetti}, where we have a first extension of Turaev's theory to the case of vector fields with simple boundary tangencies.

\begin{ack}
This paper results from the elaborations of my master
degree thesis at the University of Pisa. I thank my advisor Riccardo Benedetti for having suggested me this question and for several valuable discussions during the preparation of the paper.
\end{ack}

\section{Euler structures}\label{sctn1}

We  consider generic vector fields on a 3-manifold $M$ and we show that their behavior on the boundary $\partial M$ is fixed by the choice of a boundary pattern $\mathcal{P}$. We define the sets $\mathfrak{Eul}^c(M,\mathcal{P})$ of combinatorial Euler structures (equivalence classes of singular integer 1-chains) and $\mathfrak{Eul}^s(M,\mathcal{P})$ of smooth Euler structures (equivalence classes of generic vector fields). We describe the action of the first integer homology group $H_1(M)$ and the construction of the equivariant bijection $\Psi:\mathfrak{Eul}^c(M,\mathcal{P})\rightarrow\mathfrak{Eul}^s(M,\mathcal{P})$.

\subsection{Generic vector fields}\label{Generic vector fields}

We first introduce the object of our investigation:
\begin{ntz}
In what follows, with the word \emph{3-manifold} we will always understand a smooth compact oriented manifold of dimension 3.
\end{ntz}
Let $M$ be a 3-manifold and $\mathfrak{v}$ a non-singular vector field on $M$.
In general, there is a wide range of possible behaviors of $\mathfrak{v}$ on the boundary $\partial M$.
However, throught an easy adjustment of Whitney's results (\cite{whitney}), one can prove that, up to a small modification of the field $\mathfrak{v}$, the local models for the pair $(\partial M,\mathfrak{v})$ are the three in \figurename~\ref{boundary} only.

\begin{figure}[b]
\centering
\includegraphics[scale=0.7]{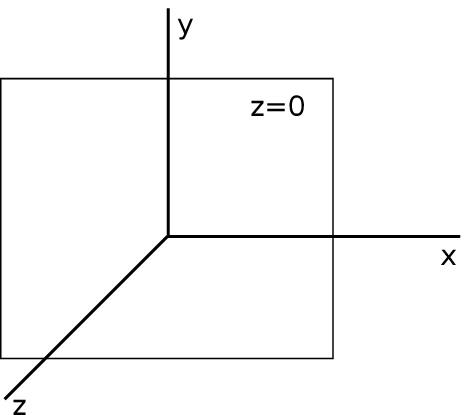}%
\qquad
\includegraphics[scale=0.7]{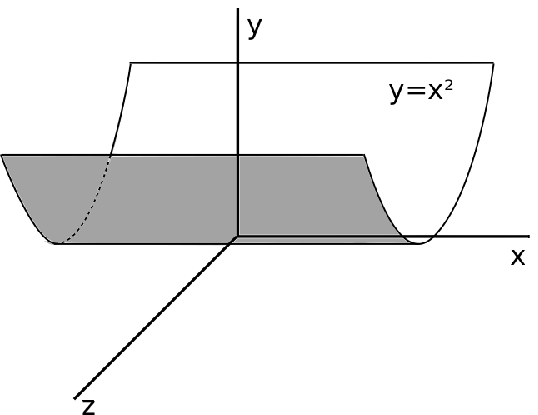}%
\qquad
\includegraphics[scale=0.7]{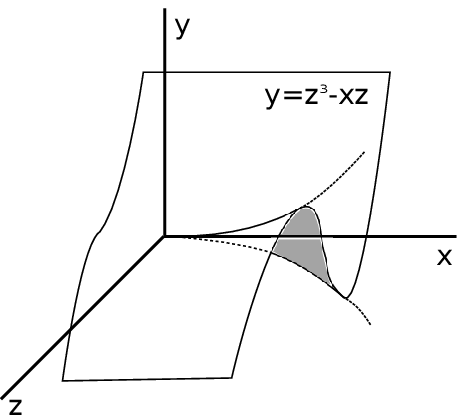}
\caption{Possible configurations of the boundary in a neighborhood of a point $p\in\partial M$. The coordinates are chosen in such a way that $p$ coincides with the origin and the vector field is headed in the $z$ direction.}\label{boundary}
\end{figure}

\begin{figure}[b]
\centering
\includegraphics[scale=1]{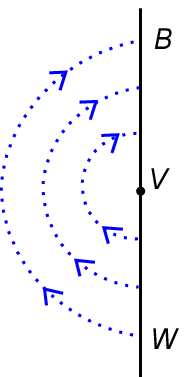}%
\qquad\qquad
\includegraphics[scale=1]{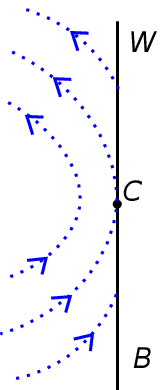}
\caption{Convex (on the left) and concave (on the right) points on the boundary.}\label{convex-concave}
\end{figure}

Therefore, given a non-singular vector field $\mathfrak{v}$ on $M$, $\mathfrak{v}$ can be slightly modified to obtain a new vector field  with the following properties:

\begin{enumerate}
\item $\mathfrak{v}$ is still non-singular on $M$;

\item $\mathfrak{v}$ is transverse to $\partial M$ in each point, except for an union $G\subset\partial M$ of circles, in which $\mathfrak{v}$ is tangent to $\partial M$;

\item $\mathfrak{v}$ is tangent to $G$ in a finite set $Q$ of points only. 
\end{enumerate}

A vector field on $M$ satisfying conditions 1,2,3 is called \emph{generic}.
A generic vector field $\mathfrak{v}$ induces a partition $\mathcal{P}=(W,B,V,C,Q^+,Q^-)$ on $\partial M$ where:

\begin{itemize}
\item $W\cup B=\partial M\setminus G$ is the set of \emph{regular} points (\figurename~\ref{boundary}-left), i.e. the points in which $\mathfrak{v}$ is transverse to $\partial M$. $W$ is the \emph{white part}, i.e., the set of the points in $\partial M$ for which $\mathfrak{v}$ is directed inside $M$;
$B$ is the \emph{black part}, i.e., the set of the points in $\partial M$ for which $\mathfrak{v}$ is directed outside $M$. 
$W$ and $B$ are interior of compact surfaces embedded in $\partial M$, and $\partial W=\partial B=G$.

\item $V\cup C=G\setminus Q$ is the set of \emph{fold} points (\figurename~\ref{boundary}-center). $V$ is the \emph{convex} part, i.e., the set of points in $G$ for which $\mathfrak{v}$ is directed towards $B$;
$C$ is the \emph{concave} part, i.e., the set of points in $G$ for which $\mathfrak{v}$ is directed towards $W$. The names (convex and concave) are justified by the cross-section in \figurename~\ref{convex-concave}. 
$V$ and $C$ are disjoint unions of circles and open segments, and $\partial V=\partial C=Q$.

\item $Q^+\cup Q^-=Q$ is the set of \emph{cuspidal} points (\figurename~\ref{boundary}-right). $Q^+$ is the set of points where $\mathfrak{v}$ is directed towards $C$; $Q^-$ is the set of points where $\mathfrak{v}$ is directed towards $V$.
\end{itemize}

Such a partition $\mathcal{P}$ is called a \emph{boundary pattern} on $\partial M$. 
A generic vector field $\mathfrak{v}$ and a boundary pattern $\mathcal{P}$ are said to be \emph{compatible} if $\mathcal{P}$ is induced by $\mathfrak{v}$, up to a diffeomorphism of $M$.

\begin{nta}
A more general work, due to Morin (\cite{morin}), generalizes the results of Whitney in every dimension. This would probably allow to extend results of Sections \ref{sctn1}, \ref{sctn2}, \ref{sctn3} to dimensions greater than 3.
\end{nta}

\subsection{Euler structures}


A combing is a pair $[M,\mathfrak{v}]$, where $M$ is a 3-manifold and $\mathfrak{v}$ is a generic vector field on $M$, viewed up to diffeomorphism of $M$ and homotopy of $\mathfrak{v}$. We denote by $\mathfrak{Comb}$ the set of all combings. Notice that, under a homotopy of $\mathfrak{v}$, the boundary pattern on $\partial M$ changes by an isotopy. Therefore, to a combing $[M,\mathfrak{v}]$ is associated a pair $(M,\mathcal{P})$ viewed up to diffeomorphism of $M$, and $\mathfrak{Comb}$ naturally splits as the disjoint union of subsets $\mathfrak{Comb}(M,\mathcal{P})$ of combings on $M$ compatible with $\mathcal{P}$. 

Two classes  $[M,\mathfrak{v}_1], [M,\mathfrak{v}_2]\in\mathfrak{Comb}(M,\mathcal{P})$ are said to be \emph{homologous} if $\mathfrak{v}_1,\mathfrak{v}_2$ are obtained from each other by homotopy throught vector fields compatible with $\mathcal{P}$ and modifications supported into closed interior balls (that is, up to homotopy, $\mathfrak{v}_1, \mathfrak{v}_2$ coincide outside a ball contained in $\mathrm{Int}\, M$).

The quotient of $\mathfrak{Comb}(M,\mathcal{P})$ throught the equivalence relation of homology is denoted by $\mathfrak{Eul}^s(M,\mathcal{P})$, and its elements are called \emph{smooth Euler structures}.

\begin{prp}\label{non-emptyness2}
$\mathfrak{Eul}^s(M,\mathcal{P})$ is non-empty if and only if $\chi(M)-\chi(W)-\chi(V)-\chi(Q^+)=0$.
\end{prp}

\begin{proof}
This result will be an immediate consequence of Proposition~\ref{non-emptyness} and Theorem~\ref{reconstruction map} below.
A direct proof can be enstablished in a way similar to \cite[Prop.~1.1]{benedetti}, as an application of the Hopf theorem. 
\end{proof}

Let $H_1(M)$ be the first integer homology group of $M$.
It is a standard fact of obstruction theory (see \cite[\S~5.2]{turaev} for more details) that the map 
$$\alpha^s:\mathfrak{Eul}^s(M,\mathcal{P})\times\mathfrak{Eul}^s(M,\mathcal{P})\rightarrow H_1(M),
$$
which associates to a pair $(\mathfrak{e_1},\mathfrak{e}_2)$ the first obstruction $\alpha_s(\mathfrak{e_1},\mathfrak{e}_2)\in H_1(M)$ to their homotopy, is well defined. The map $\alpha^s$ defines an action of $H_1(M)$ on $\mathfrak{Eul}^s(M,\mathcal{P})$.
\\

Recall that every 3-manifold admits a cellularization; this is a consequence of the Hauptvermutung or of Theorem~\ref{whitehead} below.
A finite cellularization $\mathcal{C}$ of $M$ is called \emph{suited} to $\mathcal{P}$ if points in $Q^+$ and $Q^-$ are 0-cells of $\mathcal{C}$ and $G=V\cup C\cup Q^+\cup Q^-$ is a subcomplex. Let such a $\mathcal{C}$ be given. 
Denote by $E_{\mathcal{C}}$ the union of the cells of $M\setminus(W\cup V\cup Q^+)$.
An \emph{Euler chain} is an integer singular 1-chain $\xi$ in $M$ such that
\begin{equation}\label{Euler chains}
\partial\xi=\sum_{e\subset E_{\mathcal{C}}} (-1)^{\dim(e)}\cdot x_e
\end{equation}
where $x_e\in e$ for all $e$.

Given two Euler chains $\xi,\xi'$ with boundaries $\partial\xi=\sum(-1)^{\dim (e)} x_e$, $\partial\xi'=\sum (-1)^{\dim (e)} y_e$, we say that $\xi, \xi'$ are \emph{homologous} if, chosen for each $e\in E_{\mathcal{C}}$ a path $\alpha_e$ from $x_e$ to $y_e$, the $1$-cycle
$$\xi-\xi'+\sum_{e\in E_{\mathcal{C}}} (-1)^{\dim (e)} \alpha_e$$
represents the class $0$ in $H_1(M)$.

Define $\mathfrak{Eul}^c(M,\mathcal{P})_{\mathcal{C}}$ as the set of homology classes of Euler chains. The following result was proved by Turaev (see \cite[\S~1.2]{turaev}) in his framework, but extends to our setting without significant modifications.

\begin{prp}
If $\mathcal{C'}$ is a subdivision of $\mathcal{C}$, then there exists a canonical $H_1(M)$-isomorphism $\mathfrak{Eul}^c(M,\mathcal{P})_{\mathcal{C}}\rightarrow\mathfrak{Eul}^c(M,\mathcal{P})_{\mathcal{C'}}$.
\end{prp}

Thus, the set $\mathfrak{Eul}^c(M,\mathcal{P})$ is canonically defined up to $H_1(M)$-isomorphism, independently of the cellularization. The elements of $\mathfrak{Eul}^c(M,\mathcal{P})$ are called \emph{combinatorial Euler structure} of $M$ compatible with $\mathcal{P}$.

\begin{prp}\label{non-emptyness}
$\mathfrak{Eul}^c(M,\mathcal{P})$ is non-empty if and only if $\chi(M)-\chi(W)-\chi(V)-\chi(Q^+)=0$.
\end{prp}

\begin{proof}
Follows immediately from the observation that the algebraic number of points appearing on the right side of \eqref{Euler chains} is $\chi(M)-\chi(W)-\chi(V)-\chi(Q^+)$.
\end{proof}

It is easy to obtain an action of $H_1(M)$ on $\mathfrak{Eul}^c(M,\mathcal{P})$: it is the one induced by the map
$$
\alpha^c:\mathfrak{Eul}^c(M,\mathcal{P})\times\mathfrak{Eul}^c(M,\mathcal{P})\rightarrow H_1(M) 
$$
defined by $\alpha^c(\mathfrak{e}_1,\mathfrak{e}_2)= [\mathfrak{e}_1 - \mathfrak{e}_2]$.

\subsection{Reconstruction map}

A fundamental result is that combinatorial and differentiable approach are equivalent, as stated by the following theorem.

\begin{teo}\label{reconstruction map}
There exists a canonical $H_1(M)$-equivariant isomorphism
$$
\Psi:\mathfrak{Eul}^c(M,\mathcal{P})\rightarrow\mathfrak{Eul}^s(M,\mathcal{P})
$$
\end{teo}

The map $\Psi$ in the theorem is called \emph{reconstruction map}, and it is explicitly constructed in the proof of the theorem.

\begin{proof}
The proof follows the scheme of \cite[Thm.~1.4]{benedetti}, which in turn is an extension of \cite[\S~6]{turaev}. 

Let $M'$ be the manifold obtained by attaching the collar $\partial M\times[0,+\infty)$ along $\partial M$, in such a way that $\partial M\times \{0\}$ is identified with $\partial M$. Consider a cellularization $\mathcal{C}$ of $M$: $\mathcal{C}$ extends to a ``cellularization'' $\mathcal{C}'$ on $M'$ by attaching a cone to every cell of $\partial M$ and then removing the vertex.
Notice that some of the cells of $\mathcal{C}'$ have ideal vertices, thus $\mathcal{C}'$ is not a proper cellularization. 

We set the following hypotheses:
\begin{enumerate}
\item[(Hp1)] $\mathcal{C}$ is suited with $\mathcal{P}$;
\item[(Hp2)] $\mathcal{C}$ is obtained by face-pairings on a finite number of polyhedra, and the projection of each polyhedron to $M$ is smooth.
\end{enumerate}
Such a cellularization certainly exists: for instance, a triangulation $\mathcal{T}$ of $M$ satisfy (Hp2), and up to subdivision we can suppose that $\mathcal{T}$ is suited with $\mathcal{P}$.

For a cellularization satisfying (Hp2), one can recover the  ``first barycentric subdivision'' of $\mathcal{C}'$, that will be denoted by $\mathcal{C}''$. Its vertices are the points $\{p_{\sigma}\}_{\sigma\in\mathcal{C}'}$, where $p_{\sigma}$ is inside the open cell $\sigma$ for all $\sigma\in\mathcal{C}'$. Moreover, it is well defined a canonical vector field $\mathfrak{w}_{\mathcal{C}'}$ with the following properties:
\begin{itemize}
\item $\mathfrak{w}_{\mathcal{C}'}$ has singularities, of index $(-1)^{\dim\sigma}$, in the points $p_{\sigma}$ only;
\item the orbits of $\mathfrak{w}_{\mathcal{C}'}$ start (asintotically) from a point $p_{\sigma}$ and end (asintotically) in a point $p_{\sigma'}$ with $\sigma'\subset\sigma$.
\end{itemize}
\figurename~\ref{ffield} shows the behavior of $\mathfrak{w}_{\mathcal{C}'}$ on a triangle. The exact definition of $\mathfrak{w}_{\mathcal{C}'}$ is given in \cite{halperin} for triangulations, but extends to our cellularization without complications.
From now on, $\mathfrak{w}_{\mathcal{C}'}$ will be called \emph{fundamental field} of the cellularization $\mathcal{C}'$.

\begin{figure}

\centering

\includegraphics[scale=0.7]{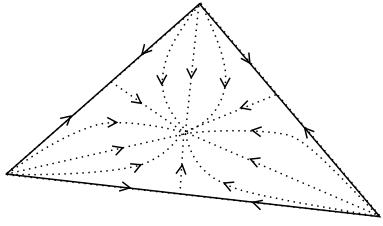}

\caption{$\mathfrak{w}_{\mathcal{T}'}$ on a triangle.}\label{ffield}

\end{figure}

Let $N$ be the star of $\partial M$ in $\mathcal{C}''$; identify $N$ with $\partial M\times (-1,1)$, in such a way that $\partial M\times (-1,0]= M\cap N$.

Given a map $h:\partial M\rightarrow (-1,1)$, we denote by $M_h$ the  manifold 
$$
M_h=(M\setminus N)\cup\{(x,t)\in\partial M\times(-1,1)\cong N:t\leq h(x)\}.
$$
Notice that $M$ and $M_{h}$ are isomorphic. We want to choose $h$ in such a way that $\mathfrak{w}_{\mathcal{C}'}$ has no singularities on $\partial M_h$ and induces on $\partial M_h\cong\partial M$ the boundary pattern $\mathcal{P}$.

Let $Q=Q^+\cup Q^-$ and $G=V\cup C\cup Q$ (recall that $G$ is a disjoint union of circles and $Q\subset G$ is a finite union of points).
Denote by $U\subset \partial M$ the star of $G$ in $\mathcal{C}''_{\partial}$ (where $\mathcal{C}''_{\partial}$ is the restriction of $\mathcal{C}''$ to $\partial M$). We have a diffeomorphism $U\cong G\times (-1,1)$ such that:
$$
G\cong G\times\{0\}\quad ;\quad U\cap W\cong G\times (-1,0)\quad ;\quad U\cap B\cong G\times(0,1)
$$
On $\partial M\setminus U$ we define $h$ by:
$$
h(p)=
\begin{cases}
-\frac{1}{2}            & \mbox{, if } p\in W\setminus U\\
\frac{1}{2}  & \mbox{, if } p\in B\setminus U
\end{cases}
$$
Obviously $\mathfrak{w}_{\mathcal{C}'}$ points outside $M_h$ on $W$ and inside $M_h$ on $B$, as wished. 

It remains to define $h$ on $U$. To simplify the exposition, we are going to make the following hypothesis on the cellularization $\mathcal{C}$:
\begin{enumerate}
\item[(Hp3)] The star in $\mathcal{C}'$ of each 0-cell $p$ is formed by eight 3-cells, arranged in such a way that the star in $\mathcal{C}''$ of $p$ has the form shown in \figurename~\ref{functiong}-left.
\end{enumerate}
It is clear that a cellularization satisfying (Hp1), (Hp2), (Hp3) exists: again, one starts from a triangulation $\mathcal{T}$ of $M$ suited with $\mathcal{P}$. By unifying or subdividing some of the simplices of $\mathcal{T}$, one obtains a cellularization (that still satisfies (Hp1), (Hp2)) such that the star in $\mathcal{C}'$ of each 0-cell $p$ is formed by four 3-cells, disposed in the right way (namely, the boundary of each 3-cell does not contain both the convex and concave line incident in $p$). The extension of this cellularization to a ``cellularization'' $\mathcal{T}'$ of $M'$ satisfies (Hp3).

We also need to define a preliminary continuous function $g:[-\frac{1}{3},\frac{1}{3}]\times [-1,1]\rightarrow [-\frac{1}{2},\frac{1}{2}]$ as follows. Consider the square $[-1,1]\times[-1,1]$ and the fundamental field of its obvious cellularization (4 vertices, 4 edges and one 2-cell). For $\bar x\in [-\frac{1}{3},\frac{1}{3}]\setminus \{0\}$, we impose the one-variable function $g_{\bar x}(t)=g(\bar x,t)$ to be an increasing function with all derivatives zero in $-1,1$, with $g_{\bar x}(1)=\frac{1}{2}, g_{\bar x}(-1)=-\frac{1}{2}$ and with the property that the fundamental field is tangent to the curve $t\mapsto (t,g_{\bar x}(t))$  for $t=\bar x$ only. 
For $\bar x=0$: $g_0(t)=g(0,t)$ is a strictly increasing function with all derivatives zero in $-1,1$, with $g_0(1)=\frac{1}{2}, g_0(-1)=-\frac{1}{2},g_0(0)=0$ and never tangent to the fundamental field.
It is clear that such a function $g$ exists: we show $g_{\bar x}$ in \figurename~\ref{functiong}-center,right. We will avoid its explicit construction, that is not very significant.

\begin{figure}

\centering

\includegraphics[scale=0.35]{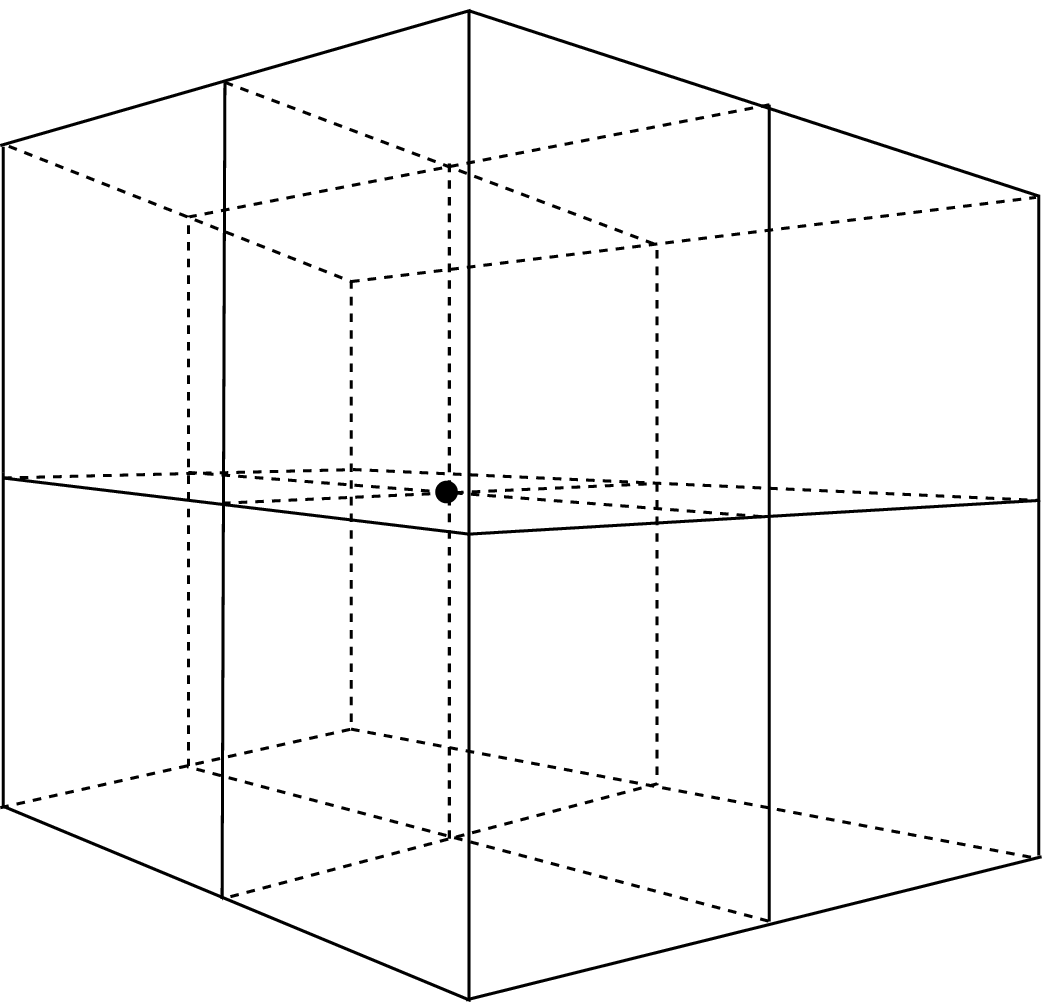}%
\qquad
\includegraphics[scale=0.55]{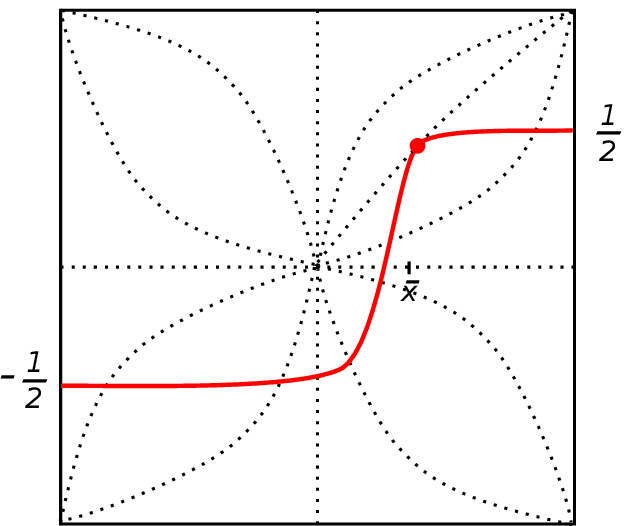}%
\qquad
\includegraphics[scale=0.55]{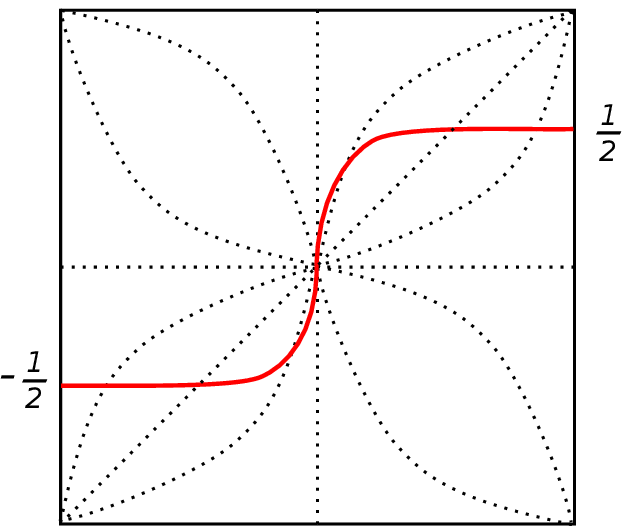}

\caption{The star of a vertex in $G$ (left); the function $g_{\bar x}$ (center) and $g_0$ (right).}\label{functiong}

\end{figure}

Let $T\subset G$ be the star of $Q$ in $\mathcal{C}''_G$, where $\mathcal{C}''_G$ is the restriction of $\mathcal{C}''$ to $G$ ($T$ is just a disjoint union of segments). Let $U_V,U_C,U_Q\subset U$ be the stars of $V,C,Q$ in $\mathcal{C}''_{\partial}$. Identify $U_V,U_C,U_G$ with $V\times (-1,1), C\times (-1,1),T\times (-1,1)$ consistently with the identification of $U$ with $G\times (-1,1)$. 

On $U\setminus U_Q$, define $h$ as follows:
$$
h(s,t)=
\begin{cases}
g(\frac{1}{3},t)             & \mbox{, if } (s,t)\in U_V\setminus U_Q\cong (V\setminus T)\times (-1,1) \\
g(-\frac{1}{3},t)            & \mbox{, if } (s,t)\in U_C\setminus U_Q\cong(C\setminus T)\times (-1,1).
\end{cases}
$$
It is clear that $h$ induces the wished pattern: to the points of $V\setminus T$ corresponds a convex point in $\partial M_h$ (\figurename~\ref{nearconvex}), while to the points in $C\setminus T$ corresponds a concave point in $\partial M_h$ (\figurename~\ref{nearconcave}).

\begin{figure}

\centering

\includegraphics[scale=0.6]{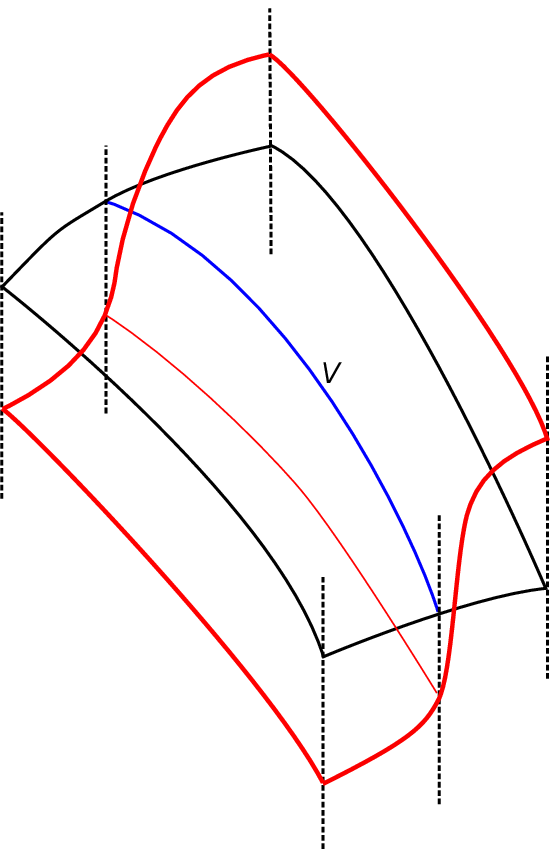}%
\qquad\qquad
\includegraphics[scale=0.6]{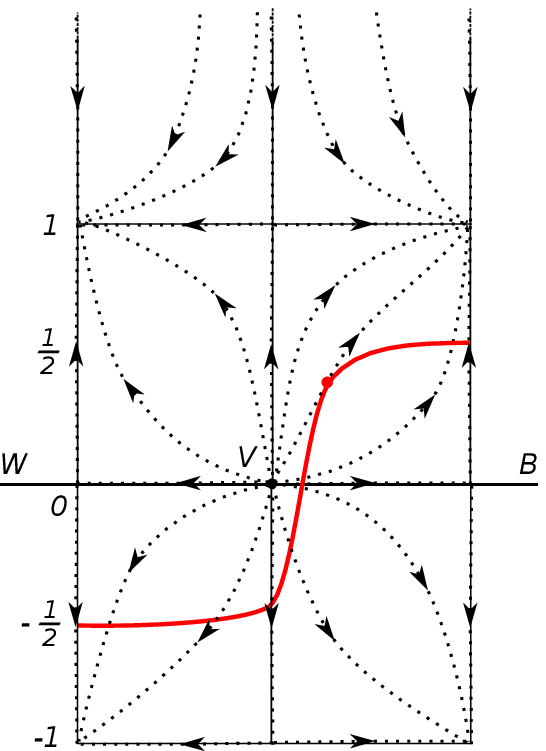}

\caption{The field $\mathfrak{w}_{\mathcal{C}'}$ has convex tangency on $V$.}\label{nearconvex}

\end{figure}

\begin{figure}

\centering

\includegraphics[scale=0.6]{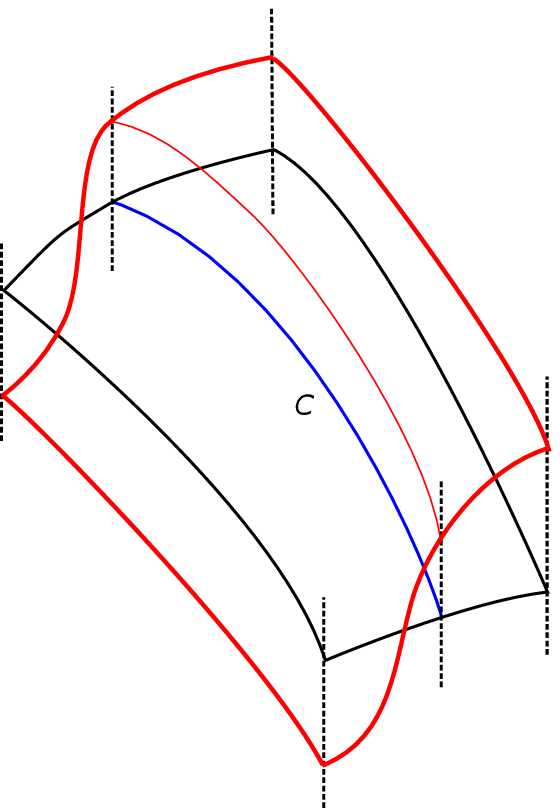}%
\qquad\qquad
\includegraphics[scale=0.6]{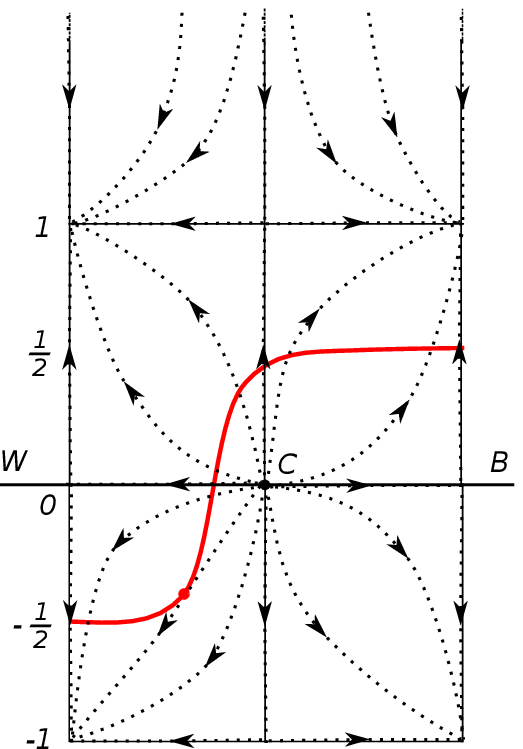}

\caption{The field $\mathfrak{w}_{\mathcal{C}'}$ has concave tangency on $C$.}\label{nearconcave}

\end{figure}

It only remains to define $h$ on $U_Q\cong T\times (-1,1)$. Identify each connected component of $T$ with $(-1,1)$ in such a way that $(-1,0)\subset C$, $(0,1)\subset V$. Now each connected component of $U_Q$ is identified with the square $(-1,1)\times(-1,1)$ in such a way that:
$$
U_Q\cap W\cong (-1,1)\times (-1,0)\quad ;\quad U_Q\cap B\cong (-1,1)\times(0,1)\quad;
$$
$$
 U_Q\cap U_C\cong (-1,0)\times (-1,1)\quad ;\quad U_Q\cap U_V\cong (0,1)\times (-1,1).
$$
If $U_{Q^+},U_{Q^-}$ are the stars of $Q^+,Q^-$ in $\mathcal{C}'_{\partial}$, we have $U_Q=U_{Q^+}\cup U_{Q^-}$.
Set:
$$
h(s,t)=
\begin{cases}
g\left(\frac{2}{3}\, g(-\frac{1}{3},s),t\right)             & \mbox{, if } (s,t)\in U_{Q^+}\cong(-1,1)\times (-1,1)\\
g\left(\frac{2}{3}\, g(\frac{1}{3},s),t\right)            & \mbox{, if } (s,t)\in U_{Q^-}\cong (-1,1)\times (-1,1).
\end{cases}
$$
The behavior of $h$ near $Q^+$ and $Q^-$ is described in \figurename~\ref{nearcusps} and \figurename~\ref{nearcuspscr}. It is clear that we can choose $g$ in such a way that $\mathfrak{w}_{\mathcal{T}'}$ is tangent to $\partial M_h$ only on the green and yellow line (the best way to convince ourself about it is by choosing $g$ in such a way that $g_{\bar x}$ is everywhere constant, except in a small neighborhood of $\bar x$, where it quickly increase from $-\frac{1}{2}$ to $\frac{1}{2}$). Notice that to each point in $Q^+$ corresponds a positive cuspidal point and to each point in $Q^-$ corresponds a negative cuspidal point.
Now $h$ is a smooth function defined on all $\partial M$ and $\mathfrak{w}_{\mathcal{T}'}$ induces the wished partition $\mathcal{P}$ on $\partial M_h$.

\begin{figure}

\centering

\includegraphics[scale=0.6]{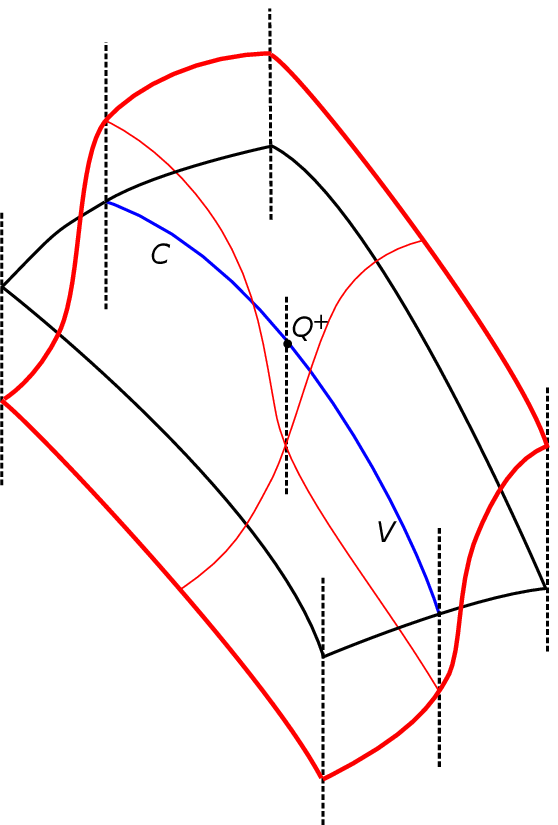}%
\qquad\qquad
\includegraphics[scale=0.7]{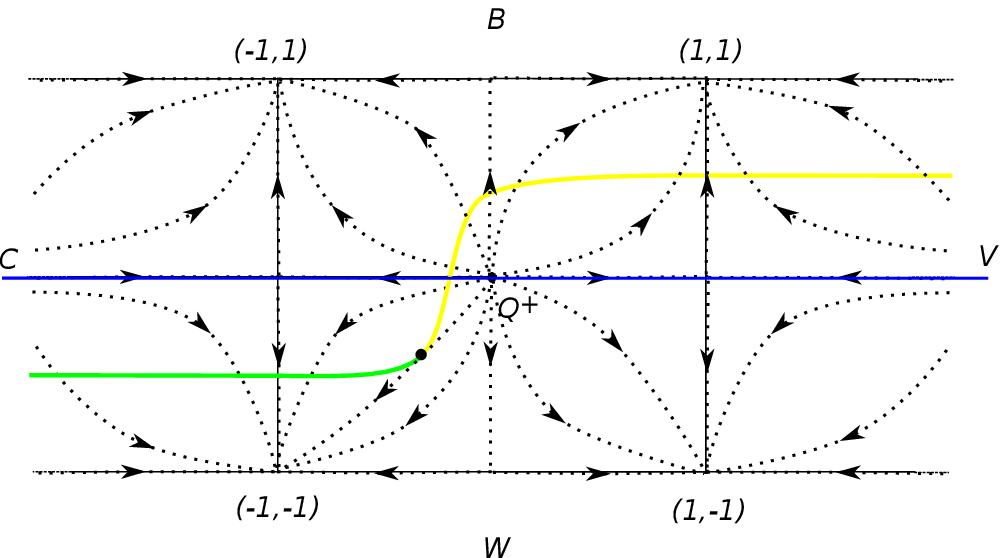}

\caption{Behavior of $\partial M_h$ near $Q^+$ (left). View from above of $\mathfrak{w}_{\mathcal{C}'}$ near $Q^+$ (right): the field goes from a concave tangency line (green) to a convex tangency line (yellow) throught a positive cuspidal point.}\label{nearcusps}

\end{figure}

\begin{figure}

\centering

\includegraphics[scale=0.6]{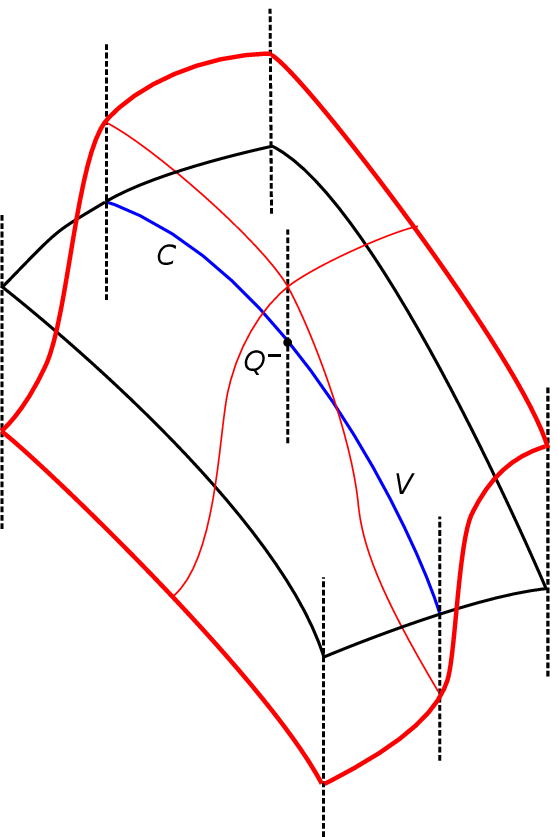}%
\qquad\qquad
\includegraphics[scale=0.7]{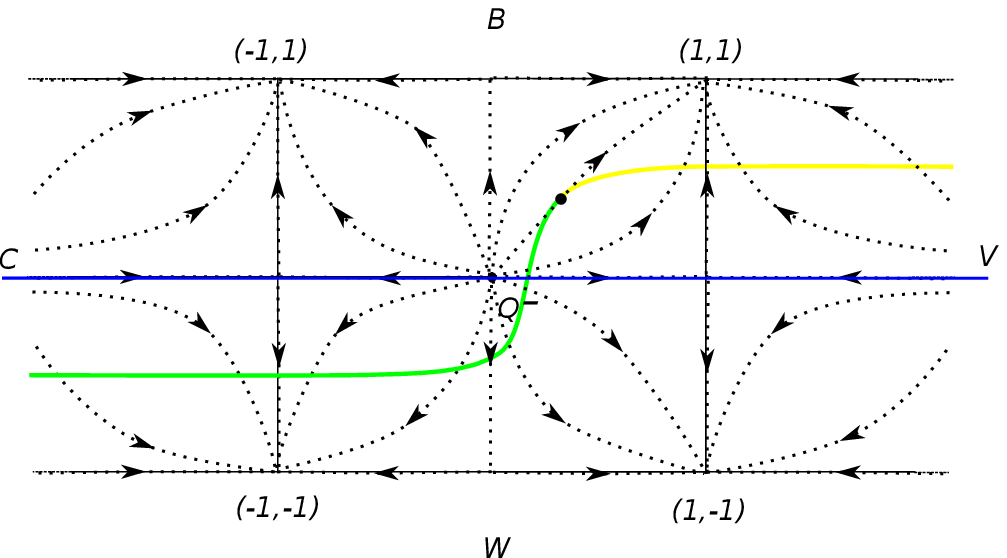}

\caption{Behavior of $\partial M_h$ near $Q^-$ (left). View from above of $\mathfrak{w}_{\mathcal{C}'}$ near $Q^-$ (right): the field goes from a concave tangency line (green) to a convex tangency line (yellow) throught a negative cuspidal point.}\label{nearcuspscr}

\end{figure}

Remember that $\mathfrak{w}_{\mathcal{C}'}$ has singularities in the 0-cells $p_{\sigma}$ of $\mathcal{C}''$.
Consider a combinatorial Euler structure $\mathfrak{e}^c\in\mathfrak{Eul}^c(M,\mathcal{P})$, and a representative $\xi$ of $\mathfrak{e}^c$. We can suppose that
$$
\partial\xi=\sum_{\sigma\in E_{\mathcal{C}}} (-1)^{\dim\sigma}\cdot p_{\sigma},
$$
where $E_{\mathcal{C}}$ is the union of the cells of $\mathcal{C}$ in $M\setminus (W\cup V\cup Q^+)$.
Notice that $\partial\xi$ consists exactly of the singularities of $\mathfrak{w}_{\mathcal{C}'}$ in $M_h$, each taken with its index.
Moreover the sum of the indices of these singularities is zero, hence it is possible to modify the field on a neighborhood of the support of $\xi$ in order to remove them. In this way, we obtain a non-singular vector field $\mathfrak{w}_{\mathcal{C}'}^{\xi}$ on $M_h\cong M$, representing a smooth Euler structure $\Psi(\mathfrak{e}^c)\in\mathfrak{Eul}^s(M,\mathcal{P})$. Turaev's proof that $\Psi$ is well defined and $H_1(M)$-equivariant extends to our case without particular modifications. The $H_1(M)$-equivariance proves the bijectivity of $\Psi$.
\end{proof}

\begin{nta}
The bijectivity of $\Psi$ is obtained indirectly from the $H_1(M)$-equivariance, while the explicit construction of the inverse $\Psi^{-1}$ is an harder task. 
In Section \ref{sctn4} we will see how to invert $\Psi$ using stream-spines. 
\end{nta}

\begin{nta}\label{fund field}

While hypotheses (Hp1), (Hp2) on the cellularization $\mathcal{C}$ are necessary, the hypothesis (Hp3) is not fundamental. The proof above can be repeated without using (Hp3): in the construction of $h$ inside $U_Q$, one have to distinguish various cases, depending on the form of the stars of the cuspidal points.
\end{nta}

\begin{ntz}
Theorem~\ref{reconstruction map} allows us to ease the notation: if there is no ambiguity, we will write $\mathfrak{Eul}(M,\mathcal{P})$ to denote either $\mathfrak{Eul}^c(M,\mathcal{P})$ or $\mathfrak{Eul}^s(M,\mathcal{P})$; $\alpha$ to denote either $\alpha^c$ or $\alpha^s$.
\end{ntz}

\section{Reidemeister torsion}\label{sctn2}

Definitions in Sections \ref{Chain complexes} and \ref{Twisted homology} are known facts, preparatory to Section \ref{Torsion of a 3-Manifold}, where Reidemeister torsion of a pair $(M,\mathcal{P})$ is defined. The main result is Proposition~\ref{Reid Tors}, which shows that the ambiguity in the definition of Reidemeister torsion is fixed (up to sign) by the choice of an Euler structure $\mathfrak{e}\in\mathfrak{Eul}(M,\mathcal{P})$.

\subsection{Torsion of a chain complex}\label{Chain complexes}

Consider a finite chain complex over a field $\mathbb{F}$
$$C=\left(C_m\xrightarrow{\partial_m} C_{m-1}\xrightarrow{\partial_{m-1}}\cdots\xrightarrow{\partial_2} C_1\xrightarrow{\partial_1} C_0\right).$$
By \emph{finite} we mean that every vector space $C_i$ has finite dimension. We fix bases $\mathfrak{c}=(\mathfrak{c}_0,\dots,\mathfrak{c}_m)$ and $\mathfrak{h}=(\mathfrak{h}_0,\dots,\mathfrak{h}_m)$ of $C$ and $H_*(C)$ respectively. With this notation, we mean that $\mathfrak{c}_i$ (resp. $\mathfrak{h}_i$) is a basis of $C_i$ (resp. $H_i(C)$) for all $i=0,\dots,m$.

For all $i=0,\dots,m$, we choose an arbitrary basis $\mathfrak{b}_i$ of the $i$-boundaries $B_i=\mbox{Im}(\partial_{i+1})$, and we consider the short exact sequence:
\begin{equation}\label{ses1}
0\rightarrow B_i\rightarrow Z_i \rightarrow H_i(C) \rightarrow 0
\end{equation}
where $Z_i=\mbox{Ker}(\partial_i)$ is the group of the $i$-cycles.
By inspecting sequence \eqref{ses1}, it is clear that a basis for $Z_i$ is obtained by taking the union of $\mathfrak{b}_i$ and a lift $\widetilde{\mathfrak{h}}_i$ of $\mathfrak{h}_i$.

Now, consider the exact sequence:
\begin{equation}\label{ses2}
0\rightarrow Z_i \rightarrow C_i \xrightarrow{\partial_i} B_{i-1} \rightarrow 0
\end{equation}

By \eqref{ses2}, $\mathfrak{b}_i\widetilde{\mathfrak{h}}_i\widetilde{\mathfrak{b}}_{i-1}$ (where $\mathfrak{b}_i\widetilde{\mathfrak{h}}_i$ is the basis of $Z_i$ constructed above and $\widetilde{\mathfrak{b}}_{i-1}$ is a lift of $\mathfrak{b}_{i-1}$) is a basis for $C_i$.

We can define the torsion of the chain complex $C$ as a sort of difference between the basis $\mathfrak{c}$ and the new basis of $C$ obtained above.

\begin{ntz}
Given two bases $\mathfrak{A},\mathfrak{B}$ of the finite-dimensional vector space $V$, we denote by $[\mathfrak{A}/\mathfrak{B}]$ the determinant of the matrix that represents the change of basis from $\mathfrak{A}$ to $\mathfrak{B}$ (i.e., the matrix whose columns are the vectors of $\mathfrak{A}$ written in coordinates with respect to the basis $\mathfrak{B}$). 
\end{ntz}

The \emph{torsion} of the chain complex $C$ is defined by:
\begin{equation}\label{dfn:torsion}
\tau(C;\mathfrak{c},\mathfrak{h})=\prod_{i=0}^m [\mathfrak{b}_i\widetilde{\mathfrak{h}}_i\widetilde{\mathfrak{b}}_{i-1}/\mathfrak{c}_i]^{(-1)^{i+1}}\in\mathbb{F}
\end{equation}

The torsion $\tau(C,\mathfrak{c},\mathfrak{h})$ depends uniquely on the equivalence classes of the bases $\mathfrak{c}_i,\mathfrak{h}_i$, and does not depend on the choice of $\mathfrak{b}_i$ and of the lifts $\widetilde{\mathfrak{b}}_i,\widetilde{\mathfrak{h}}_i$.

The following is an interesting result, that will be fundamental in Section \ref{sctn3}.

\begin{teo}[Milnor \cite{milnor1}]\label{milnor}
Consider a short exact sequence of finite complexes
$$
0\rightarrow C'\rightarrow C\rightarrow C''\rightarrow 0
$$
and the corresponding long exact sequence in homology
$$
\mathcal{H}=\left(H_m(C')\rightarrow H_m(C)\rightarrow\cdots\rightarrow H_0(C)\rightarrow H_0(C'')  \right).
$$
 $\mathcal{H}$ can be viewed as a finite acyclic chain complex. Fix bases $\mathfrak{h}',\mathfrak{h},\mathfrak{h}''$ on $H_*(C'), H_*(C), H_*(C'')$ respectively. This gives a basis on $\mathcal{H}$; denote by $\tau(\mathcal{H})$ the torsion of $\mathcal{H}$, computed with respect to this basis. Choose \emph{compatible} bases $\mathfrak{c}',\mathfrak{c},\mathfrak{c}''$ on $C',C,C''$ (by compatible, we mean that $\mathfrak{c}$ is the union of $\mathfrak{c}'$ and a lift of $\mathfrak{c}''$). With these hypothesis, the following formula holds:
$$
\tau(C;\mathfrak{c},\mathfrak{h})=\tau(\mathcal{H})\cdot\tau(C';\mathfrak{c}',\mathfrak{h}')\cdot\tau(C'';\mathfrak{c}'',\mathfrak{h}'')
$$
\end{teo}

\subsection{Torsion of a pair}\label{Twisted homology}

Let $M$ be a smooth compact oriented manifold of arbitrary dimension and let $H_1(M)$ be its first integer homology group. 
To recover the definitions of Section \ref{Chain complexes} we need a cellularization of $M$. The existence of a cellularization is granted by the following classical result:

\begin{teo}[Whitehead]\label{whitehead}
Every compact smooth manifold $M$ admits a canonical PL-structure (in particular, $M$ admits a cellularization, unique up to subdivisions).
\end{teo}

Let $N$ be a compact submanifold of $M$. Consider a cellularization $\mathcal{C}$ of $M$ such that $N$ is a closed subcomplex of $M$.
Assume that $M$ is connected. If $\hat M\rightarrow M$ is the maximal abelian covering, $\mathcal{C}$ lifts to a cellularization of $\hat M$. Notice that $p^{-1}(N)$ is a closed subcomplex of $\hat M$, hence we can consider the cellular chain complex
$$
C_*(M,N)=C_*^{\mathrm{cell}}(\hat M, p^{-1}(N);\mathbb{Z})
$$
$H_1(M)$ acts on $C_*(M,N)$ via the deck transformations, thus $C_*(M,N)$ can be viewed as a chain complex of $\mathbb{Z}[H_1(M)]$-modules and $\mathbb{Z}[H_1(M)]$-homomorphisms.

Now, consider a field $\mathbb{F}$ and a \emph{representation} $\varphi$, i.e., a ring homomorphism $\varphi:\mathbb{Z}[H_1(M)]\rightarrow\mathbb{F}$. The field $\mathbb{F}$ can be viewed as a $\mathbb{Z}[H_1(M)]$-module with the product $z\cdot f=f\varphi(z)$ (where $z\in\mathbb{Z}[H_1(M)], f\in\mathbb{F}$).
Therefore we can consider the following chain complex over $\mathbb{F}$:

\begin{equation}
C_*^{\varphi}(M,N)=C_*(M,N)\otimes_{\mathbb{Z}[H_1(M)]}\mathbb{F}.
\end{equation}
$C_*^{\varphi}(M,N)$ is called \emph{$\varphi$-twisted chain complex} of $(M,N)$.
Its homology (the \emph{$\varphi$-twisted homology}) is denoted by $H_*^{\varphi}(M,N)$.

A \emph{fundamental family} $\mathfrak{f}$ of $(M,N)$ is a choice of a lift for each cell in $M\setminus N$. $\mathfrak{f}$ is a basis of $C_*^{\varphi}(M,N)$, thus, chosen a basis $\mathfrak{h}$ on $H_*^{\varphi}(M,N)$, we can compute the torsion of the twisted complex $C_*^{\varphi}(M,N)$. 

The \emph{Reidemeister torsion} of $(M,N)$ with respect to $\mathfrak{f},\mathfrak{h}$ is defined by
\begin{equation}\label{RTofthepair}
\tau^{\varphi}(M,N;\mathfrak{f},\mathfrak{h})=\tau(C_*^{\varphi}(M,N);\mathfrak{f},\mathfrak{h})\in\mathbb{F}^*.
\end{equation}
The fact that the definition of $\tau^{\varphi}(M,N;\mathfrak{f},\mathfrak{h})$ does not depends on the choice of the cellularization $\mathcal{C}$ is classical (see \cite[Lemma~3.2.3]{turaev}).

Definitions above extend in a natural way to the case of a non-connected manifold $M$. Namely, the twisted chain complex extends by direct sum on the connected components and Reidemeister torsion extends by multiplicativity.

\subsection{Torsion of a 3-manifold}\label{Torsion of a 3-Manifold}

Now we specialize on dimension 3. Consider a 3-manifold $M$, a boundary pattern $\mathcal{P}=(W,B,V,C,Q^+,Q^-)$
and a cellularization $\mathcal{C}$ of $M$ suited with $\mathcal{P}$. If $\hat M\rightarrow M$ is the maximal abelian covering, $\mathcal{C}$ lifts to a cellularization of $\hat M$. Assume that $M$ is connected (as in Section~\ref{Twisted homology}, the definitions below will extend to the non-connected case in the obvious way).

\begin{ntz}
Consider a submanifold $N$ of $M$, which is also a subcomplex with respect to the cellularization $\mathcal{C}$ (for instance, this happens if $N=\overline{W},\overline{B},\overline{V},\overline{C},Q^+,Q^-$).
Given a representation $\varphi:\mathbb{Z}[H_1(M)]\rightarrow\mathbb{F}$, we can compose it with the map $i_*:\mathbb{Z}[H_1(N)]\rightarrow \mathbb{Z}[H_1(M)]$ induced by the inclusion $i:N\hookrightarrow M$. This gives a representation on $N$, that we will still denote by $\varphi$, with a slight abuse of notation.
\end{ntz}

Consider a field $\mathbb{F}$ and a representation $\varphi:\mathbb{Z}[H_1(M)]\rightarrow\mathbb{F}$.
The \emph{$\varphi$-twisted chain complex} of $M$ relative to $\mathcal{P}$ is the chain complex over $\mathbb{F}$ defined by
\begin{equation}
C_*^{\varphi}(M,\mathcal{P})=C_*^{\varphi}\big(M,\overline{W}\big)\oplus C_*^{\varphi}\big(\overline{C},Q^+\big).
\end{equation}
Its homology is called \emph{$\varphi$-twisted homology} and it is denoted by $H_*^{\varphi}(M,\mathcal{P})$.
A basis of $H_*^{\varphi}(M,\mathcal{P})=H_*^{\varphi}(M,\overline{W})\oplus H_*^{\varphi}(\overline{C},Q^+)$ is a pair $(\mathfrak{h}',\mathfrak{h}'')$, where $\mathfrak{h}'$ is a basis of $H_*^{\varphi}(M,\overline{W})$ and $\mathfrak{h}''$ is a basis of $H_*^{\varphi}(\overline{C},Q^+)$

A \emph{fundamental family} $\mathfrak{f}$ of $(M,\mathcal{P})$ is a pair $(\mathfrak{f}',\mathfrak{f}'')$, where $\mathfrak{f}'$ is a fundamental family of the pair $(M,\overline{W})$ and $\mathfrak{f}''$ is a fundamental family of the pair $(\overline{C},Q^+)$.

$\mathfrak{f}=(\mathfrak{f}',\mathfrak{f}'')$ induces a combinatorial Euler structure on $M$ relative to $\mathcal{P}$ as follows.
Lifting the inclusion $\overline{C}\hookrightarrow M$, one obtains a map $\iota:\hat C\rightarrow\hat M$ (here we have denoted by $\hat C$ the maximal abelian covering of $\overline{C}$), equivariant with respect to the inclusion homomorphism $H_1(\overline{C})\rightarrow H_1(M)$. 
Take a point $x_0\in \hat M$ and a point $x_{\sigma}$ inside each cell $\sigma\in\mathfrak{f}'\cup \iota(\mathfrak{f}'')$. Choose paths $\beta_{\sigma}$ from $x_0$ to $x_{\sigma}$ and consider the 1-chain
$$\epsilon=\sum_{\sigma\in\mathfrak{f}} (-1)^{\dim (\sigma)}\beta_{\sigma}.$$
The projection of $\epsilon$ on $M$ is an Euler chain, thus it represents an Euler structure $\mathfrak{e}\in\mathfrak{Eul}^c(M,\mathcal{P})$.

Given an Euler structure $\mathfrak{e}\in\mathfrak{Eul}^c(M,\mathcal{P})$ and a basis $\mathfrak{h}$ of $H_*^{\varphi}(M,\mathcal{P})$, the \emph{Reidemeister torsion} of $M$ relative to $\mathcal{P}$ is
$$
\tau^{\varphi}(M,\mathcal{P};\mathfrak{e},\mathfrak{h})=\tau(C_*^{\varphi}(M,\mathcal{P});\mathfrak{f},\mathfrak{h})\in\mathbb{F}^*/\{\pm 1\}
$$
where $\mathfrak{f}$ is a fundamental family of $(M,\mathcal{P})$ that induces the Euler structure $\mathfrak{e}$.

\begin{prp}\label{Reid Tors}
$\tau^{\varphi}(M,\mathcal{P};\mathfrak{e},\mathfrak{h})$ is well defined. Namely, it does not depend on the choice of the fundamental family and of the cellularization.
Moreover:
\begin{equation}\label{formulina}
\tau^{\varphi}(M,\mathcal{P};\mathfrak{e}',\mathfrak{h})=\varphi\left(\alpha(\mathfrak{e},\mathfrak{e}')\right)\cdot\tau^{\varphi}(M,\mathcal{P};\mathfrak{e},\mathfrak{h})
\end{equation}
\end{prp}

\begin{proof}
The indipendence on the cellularization is a consequence of the independence on the cellularization of the Reidemeister torsion of the pair defined in Section~\ref{Twisted homology}. 
Formula \eqref{formulina} is easily proved by choosing representatives $\sum (-1)^{\dim\sigma}\beta_{\sigma}$, $\sum (-1)^{\dim\sigma}\beta'_{\sigma}$ of $\mathfrak{e},\mathfrak{e}'$ such that $\beta_{\sigma}'=\beta_{\sigma}$ for all $\sigma$ but one.

It remains to prove that $\tau^{\varphi}$ does not depend on the choice of the fundamental family. To this end, consider two fundamental families
$$
\mathfrak{f}_1=(\mathfrak{f}_1',\mathfrak{f}_1'')=(\{\sigma_1,\dots,\sigma_r\},\{\sigma_{r+1},\dots,\sigma_s\})
$$
$$
\mathfrak{f}_2=(\mathfrak{f}_2',\mathfrak{f}_2'')=(\{\tilde\sigma_1,\dots,\tilde\sigma_r\},\{\tilde\sigma_{r+1},\dots,\tilde\sigma_s\})
$$
inducing the same Euler structure $\mathfrak{e}$.
Suppose that $\mathfrak{f}_1$ and $\mathfrak{f}_2$ are ordered in such a way that $p(\sigma_j)=p(\tilde\sigma_j)\ \forall j=1,\dots,s$ (here we have denoted by the same letter the covering maps $p:\hat M\rightarrow M$ and $p:\hat C\rightarrow C$).

Hence, we can write $\tilde\sigma_j=h_j\sigma_j$, where $h_j\in H_1(M)$ for $j=1,\dots,r$, $h_j\in H_1(\overline{C})$ for $j=r+1,\dots, s$. Recall the inclusion morphism $i_*:H_1(M)\rightarrow H_1(\overline{C})$.
Because $\mathfrak{f}_1$ and $\mathfrak{f}_2$ induce the same Euler structure, we have:
$$
\prod_{j=1}^r h_j^{(-1)^{\dim\sigma_j}}\cdot\prod_{j=r+1}^s i_*(h_j)^{(-1)^{\dim\sigma_j}}=1\in H_1(M)
$$
Now, the result follows from the following easy computation (recall that $\mathfrak{h}=(\mathfrak{h}',\mathfrak{h}'')\in H_*^{\varphi}(M,\overline{W})\oplus H_*^{\varphi}(\overline{C},Q^+)$):
$$
\tau(C_*^{\varphi}(M,\mathcal{P});\mathfrak{f}_1,\mathfrak{h})=
$$
$$
=\tau\left(C_*^{\varphi}(M,\overline{W});\mathfrak{f}_1',\mathfrak{h}'\right)\cdot \tau\left(C_*^{\varphi}(\overline{C},Q^+);\mathfrak{f}_1'',\mathfrak{h}''\right)=
$$
\begin{equation}\notag
\begin{split}
=\varphi\Big(\prod_{j=1}^r h_j^{(-1)^{\dim \sigma_j}}\Big)\cdot \varphi\Big( i_*\Big(\prod_{j=r+1}^s h_j^{(-1)^{\dim \sigma_j}}\Big)\Big)\cdot \tau\left(C_*^{\varphi}(M,\overline{W});\mathfrak{f}_2',\mathfrak{h}'\right)\cdot&
\\
\cdot\tau\left(C_*^{\varphi}(\overline{C},Q^+);\mathfrak{f}_2'',\mathfrak{h}''\right)=
\end{split}
\end{equation}
$$
=\varphi\Big(\prod_{j=1}^r h_j^{(-1)^{\dim\sigma_j}}\cdot\prod_{j=r+1}^s i_*(h_j)^{(-1)^{\dim\sigma_j}}\Big)\cdot\tau(C_*^{\varphi}(M,\mathcal{P});\mathfrak{f}_2,\mathfrak{h})=
$$
$$
=\tau(C_*^{\varphi}(M,\mathcal{P});\mathfrak{f}_2,\mathfrak{h})\qedhere
$$
\end{proof}

\begin{nta}\label{homology orientation}
The definition of Reidemeister torsion above shows an indeterminacy in the sign, due to the arbitrariness in the choice of an order and an orientation of the fundamental family. A refinement of torsion exists: by means of an \emph{homology orientation}, one can rule out the sign indeterminacy (see \cite[\S~18]{turaev2}). We will not consider homology orientation in our work, for it will complicate much more than expected the discussion and results of Section \ref{sctn3}.
\end{nta}

\begin{nta}
For a boundary pattern $\mathcal{P}=(W,B,V,C,\emptyset,\emptyset)$, one can define an $H_1(M)$-equivariant bijection $\Theta:\mathfrak{Eul}(M,\mathcal{P})\rightarrow\mathfrak{Eul}(M,\theta(\mathcal{P}))$, where $\theta(\mathcal{P})=(W,B,V\cup C,\emptyset,\emptyset,\emptyset)$ (see \cite[\S~1.2]{benedetti}). In general, it does not exist a basis $\mathfrak{h}'$ of $H_*^{\varphi}(\overline{C})$ such that:
$$
\tau^{\varphi}(M,\mathcal{P};\mathfrak{e},\mathfrak{h}\cup\mathfrak{h}')=\tau^{\varphi}(M,\theta(\mathcal{P}),\Theta(\mathfrak{e}),\mathfrak{h}).
$$
Thus, our definition of Reidemeister torsion is not coherent with the one given in \cite[\S~2]{benedetti} (in the case of mixed concave and convex tangency circles).
\end{nta}

\begin{ntz}
If $M$ is closed, then the only boundary pattern is the trivial $\mathcal{P}_0=(\emptyset,\emptyset,\emptyset,\emptyset,\emptyset,\emptyset)$, thus we will avoid to specify it (for instance, we will write $C_*^{\varphi}(M)$ instead of $C_*^{\varphi}(M,\mathcal{P}_0)$). Notice that $\chi(M)=0$ (from Poincar\'e duality), hence propositions \ref{non-emptyness2} and \ref{non-emptyness} are automatically satisfied, i.e., the set of Euler structures $\mathfrak{Eul}(M)$ is not empty.
\end{ntz}

\section{Gluings}\label{sctn3}

We show how to naturally define gluings of Euler structures, and we develop a multiplicative gluing formula for Reidemeister torsion (Theorem~\ref{thm.gluing}).

\subsection{Gluing of Euler structures}

Let $M$ be a 3-manifold and $S\subset M$ an embedded surface, that divides $M$ into two smooth submanifolds $M_1,M_2$. Assume that $M$ is closed (but the extension to $\partial M\neq\emptyset$ is straightforward).

Consider an Euler chain $\xi_1$ on $M_1$, relative to a partition $\mathcal{P}=(W,B,V,$ $C,Q^+,Q^-)$ on $\partial M_1=S$. $\xi_1$ represents an Euler structure $\mathfrak{e}_1\in\mathfrak{Eul}^c(M,\mathcal{P})$.
Denote by $\mathcal{P}'$ the partition $(B,W,C,V,Q^-,Q^+)$ (namely, we swap black and white part, convex and concave lines, positive and negative cuspidal points). $\mathcal{P}$ and $\mathcal{P}'$ are said to be \emph{dual}.
Consider an Euler structure $\mathfrak{e}_2\in\mathfrak{Eul}^c(M,\mathcal{P}')$, represented by an Euler chain $\xi_2$. It is clear that $\xi=\xi_1+\xi_2$ is an Euler chain on $M$. Denote by $\mathfrak{e}_1\cup\mathfrak{e}_2\in\mathfrak{Eul}^c(M)$ the Euler structure represented by $\xi$. 
We have defined a \emph{gluing map}:
\begin{equation}\label{gluing1}
(\mathfrak{e}_1,\mathfrak{e}_2)\mapsto\mathfrak{e}_1\cup\mathfrak{e}_2:\mathfrak{Eul}^c(M_1,\mathcal{P})\times\mathfrak{Eul}^c(M_2,\mathcal{P}')\rightarrow\mathfrak{Eul}^c(M).
\end{equation}

It is easy to obtain a differentiable version of the gluing map. Consider Euler structures $\mathfrak{e}_1\in\mathfrak{Eul}^s(M_1,\mathcal{P})$, $\mathfrak{e}_2\in\mathfrak{Eul}^s(M_2,\mathcal{P}')$ represented by generic fields $\mathfrak{v}_1,\mathfrak{v}_2$ respectively. Up to homotopy, we can suppose that $\mathfrak{v}_1$ and $\mathfrak{v}_2$ coincide on $S$. Then $\mathfrak{v}_1$ and $\mathfrak{v}_2$ can be glued together (in a smooth way), giving a non-singular vector field $\mathfrak{v}$ on $M$. Again, denote by $\mathfrak{e}_1\cup\mathfrak{e}_2\in\mathfrak{Eul}^s(M)$ the Euler structure represented by $\mathfrak{v}$. Now we can define the differentiable analogous of map \eqref{gluing1}:
\begin{equation}\label{gluing2}
(\mathfrak{e}_1,\mathfrak{e}_2)\mapsto\mathfrak{e}_1\cup\mathfrak{e}_2:\mathfrak{Eul}^s(M_1,\mathcal{P})\times\mathfrak{Eul}^s(M_2,\mathcal{P}')\rightarrow\mathfrak{Eul}^s(M).
\end{equation}
The following lemma can be deduced directly from definitions:

\begin{lmm}
The following diagram is commutative
$$
\begin{CD}
\mathfrak{Eul}^c(M_1,\mathcal{P})\times\mathfrak{Eul}^c(M_2,\mathcal{P}')    @>{\cup}>>  \mathfrak{Eul}^c(M)\\
@VV{\Psi_1\times\Psi_2}V      @VV{\Psi}V\\
\mathfrak{Eul}^s(M_1,\mathcal{P})\times\mathfrak{Eul}^s(M_2,\mathcal{P}')    @>{\cup}>>  \mathfrak{Eul}^s(M)
\end{CD}
$$
\end{lmm}

\subsection{Setting}\label{Setting}

Again, let $M$ be a closed 3-manifold and $S\subset M$ an embedded surface, that splits $M$ into two smooth submanifolds $M_1,M_2$.

Let $\mathfrak{v}$ be a non-singular vector field on $M$ representing the Euler structure $\mathfrak{e}\in\mathfrak{Eul}^s(M)$. Consider the restrictions $\mathfrak{v}_1=\mathfrak{v}|_{M_1}$, $\mathfrak{v}_2=\mathfrak{v}|_{M_2}$. Up to a small modification of $S$ or $\mathfrak{v}$, we can suppose that $\mathfrak{v}_1$ (then also $\mathfrak{v}_2$) is generic. Let $\mathcal{P}=(W,B,V,C,Q^+,Q^-)$ be the partition induced by $\mathfrak{v}_1$ on $\partial M_1=S$. Then it is easy to check that $\mathfrak{v}_2$ induces on $S$ the dual partition $\mathcal{P}'=(B,W,C,V,Q^-,Q^+)$. Thus, $\mathfrak{v}_1$ (resp. $\mathfrak{v}_2$) represents an Euler structure $\mathfrak{e}_1\in\mathfrak{Eul}^s(M_1,\mathcal{P})$ (resp. $\mathfrak{e}_2\in\mathfrak{Eul}^s(M_2,\mathcal{P}')$), and $\mathfrak{e}=\mathfrak{e}_1\cup\mathfrak{e}_2$.

Choose bases $\mathfrak{h},\mathfrak{h}_1,\mathfrak{h}_2$ on the twisted homologies $H_*^{\varphi}(M)$, $H_*^{\varphi}(M_1,\mathcal{P})$, $H_*^{\varphi}(M_2,\mathcal{P}')$ respectively.


Fix a cellularization $\mathcal{C}$ on $M$ suited with $\mathcal{P}$, and choose a representation $\varphi: H_1(M)\rightarrow\mathbb{F}$. We have the following short exact sequences:

\begin{equation}\label{gs1}\tag{a}
0\rightarrow C_*^{\varphi}(Q)\rightarrow C_*^{\varphi}(\overline{V})\oplus C_*^{\varphi} (\overline{C})\rightarrow C_*^{\varphi}(G) \rightarrow 0
\end{equation}
\begin{equation}\label{gs2}\tag{b}
0\rightarrow C_*^{\varphi}(G)\rightarrow C_*^{\varphi} (\overline{W})\oplus C_*^{\varphi} (\overline{B})\rightarrow C_*^{\varphi}(S) \rightarrow 0
\end{equation}
\begin{equation}\label{gs3}\tag{c}
0\rightarrow C_*^{\varphi}(S)\rightarrow C_*^{\varphi} (M_1)\oplus C_*^{\varphi} (M_2)\rightarrow C_*^{\varphi}(M) \rightarrow 0
\end{equation}
\begin{equation}\label{gs4}\tag{d}
0\rightarrow C_*^{\varphi}(Q^+)\rightarrow C_*^{\varphi} (\overline{C})\rightarrow C_*^{\varphi} (\overline{C},Q^+) \rightarrow 0
\end{equation}
\begin{equation}\label{gs5}\tag{e}
0\rightarrow C_*^{\varphi}(Q^-)\rightarrow C_*^{\varphi} (\overline{V})\rightarrow C_*^{\varphi} (\overline{V},Q^-) \rightarrow 0
\end{equation}
\begin{equation}\label{gs6}\tag{f}
0\rightarrow C_*^{\varphi} (\overline{W})\rightarrow C_*^{\varphi} (M_1)\rightarrow C_*^{\varphi} (M_1,\overline{W}) \rightarrow 0
\end{equation}
\begin{equation}\label{gs7}\tag{g}
0\rightarrow C_*^{\varphi} (\overline{B})\rightarrow C_*^{\varphi} (M_2)\rightarrow C_*^{\varphi} (M_2,\overline{B}) \rightarrow 0
\end{equation}

It is clear that all the submanifolds appearing in the exact sequences above are also subcomplexes of $\mathcal{C}$ (because $\mathcal{C}$ is suited with $\mathcal{P}$), thus the twisted complexes are well defined.

Fix bases on the twisted homologies of the complexes above. We have complete freedom in the choice, except for the following requirements:
\begin{itemize}
\item The union of the bases of $H_*^{\varphi}(Q^+), H_*^{\varphi}(Q^-)$ gives the basis on $H_*^{\varphi}(Q)=H_*^{\varphi}(Q^+)\oplus H_*^{\varphi}(Q^-)$;
\item the union of the bases of $H_*^{\varphi}(M_1,\overline{W})$ and $H_*^{\varphi}(\overline{C},Q^+)$ gives the basis $\mathfrak{h}_1$ on $H_*^{\varphi}(M_1,\mathcal{P})$;
\item the union of the bases of $H_*^{\varphi}(M_2,\overline{B})$ and $H_*^{\varphi}(\overline{V},Q^-)$ gives the basis $\mathfrak{h}_2$ on $H_*^{\varphi}(M_2,\mathcal{P}')$;
\item the basis of $H_*^{\varphi}(M)$ is $\mathfrak{h}$;
\end{itemize}

Denote by $\tau_a,\tau_b,\tau_c,\tau_d,\tau_e,\tau_f,\tau_g$ the torsions of the long exact sequences of homologies induced by the short exact sequences \eqref{gs1},\eqref{gs2},\eqref{gs3},\eqref{gs4},\eqref{gs5}, \eqref{gs6},\eqref{gs7} respectively, computed with respect to the chosen bases. 

\subsection{A formula for gluings}\label{A formula for torsions}

\begin{teo}\label{thm.gluing}
In the notations of Section \ref{Setting}, the following gluing formula holds:
$$
\tau^{\varphi}(M;\mathfrak{e},\mathfrak{h})=\mathfrak{T}(\mathfrak{h},\mathfrak{h}_1,\mathfrak{h}_2)\cdot\tau^{\varphi}(M_1,\mathcal{P};\mathfrak{e}_1,\mathfrak{h}_1)\cdot\tau^{\varphi}(M_2,\mathcal{P}';\mathfrak{e}_2,\mathfrak{h}_2)
$$
where
$$
\mathfrak{T}(\mathfrak{h},\mathfrak{h}_1,\mathfrak{h}_2)=(\tau_a)^{-1}\cdot\tau_b\cdot(\tau_c)^{-1}\cdot\tau_d\cdot\tau_e\cdot\tau_f\cdot\tau_g
$$
\end{teo}

\begin{proof}
The idea of the proof is simple: we want to apply theorem \ref{milnor} on the exact sequences of Section \ref{Setting}. To this end, we need to specify compatible bases (at least up to sign) on the twisted complexes. The parenthesis ``at least up to sign'' is meaningfull: we remember that we are not considering homology orientations (see Remark \ref{homology orientation}) and we have a sign indeterminacy in the torsion. In order to consider signs, one have to track the behavior of the homology orientations and to choose bases compatible also in the sign (notice that some of the morphisms in the Mayer-Vietoris exact sequences have a minus sign); this will complicate too much the proof and the results.

We start from the exact sequence \eqref{gs4}; notice that there is only a fundamental family on $Q^+$ (because $\hat{Q}^+\cong Q^+$). Now choose a fundamental family $\mathfrak{f}_1''$ on $(\overline{C},Q^+)$. One easily sees that the union of these two fundamental families gives a fundamental family on $\overline{C}$ and these choices lead to compatible bases. 
The same approach works on \eqref{gs5}, and we obtain compatible fundamental families on $Q^-$, $(\overline{V},Q^-)$, $\overline{V}$. Denote by $\mathfrak{f}_2''$ the fundamental family on $(\overline{V},Q^-)$.

Denote by $\hat V$, $\hat C$, $\hat G$ the maximal abelian coverings of $\overline{V}$, $\overline{C}$, $G$ respectively. We have the natural inclusions $\hat V\hookrightarrow \hat G$, $\hat C\hookrightarrow \hat G$, and one easily sees that the union of the fundamental families on $(\overline{V},Q^-)$, $(\overline{C},Q^+)$ gives a fundamental family $\mathfrak{f}^G$ on $G$. If we chose on $Q$ the fundamental family given by the union of the fundamental families of $Q^+$, $Q^-$, we have that the chosen bases are compatible with respect to the exact sequence \eqref{gs1}.

Now consider the exact sequence \eqref{gs2} and the commutative diagram
$$
\begin{CD}
\hat G    @>{i_1}>>  \hat W\\
@V{i_2}VV             @V{j_1}VV\\
\hat B    @>{j_2}>>  \hat S.
\end{CD}
$$
Here $\hat W$, $\hat B$, $\hat S$ are the maximal abelian coverings of $\overline W$, $\overline B$, $S$. The morphisms are lifts of the corresponding inclusion, and they are equivariant with respect to the inclusion homomorphism in first integer homology. We already have a fundamental family $\mathfrak{f}^G$ on $G$. $i_1(\mathfrak{f}^G)$ is a family of cells in $\hat W$ such that each cell in $G\subset W$ lifts to exactly one cell in the family. Complete $i_1(\mathfrak{f}^G)$ to a  fundamental family $\mathfrak{f}^W$ on $\hat W$ by adding a lift for each cell in $W\setminus G$. In the same way, starting from the family $i_2(\mathfrak{f}^G)$, we obtain a fundamental basis $\mathfrak{f}^B$ of $B$. Notice that $\mathfrak{f}^S=j_1(\mathfrak{f}^W)\cup j_2(\mathfrak{f}^B\setminus i_2(\mathfrak{f}^G))$ is a fundamental basis of $S$ and that the bases are compatible.

It remains to analyze sequences \eqref{gs3},\eqref{gs6},\eqref{gs7}. Consider the commutative diagram
$$
\begin{CD}
\hat S    @>{r_1}>>  \hat M_1\\
@V{r_2}VV             @V{s_1}VV\\
\hat M_2    @>{s_2}>>  \hat M.
\end{CD}
$$
As above, we want to complete $r_1(\mathfrak{f}^S)$ to a fundamental family of $M_1$. 
Consider the family in $r_1(\mathfrak{f}^S)$ of the cells that are lifts of cells in $S\setminus\overline{W}=B$, and complete it to a fundamental family $\mathfrak{f}_1'$ of $(M_1,\overline{W})$ such that $\mathfrak{f}_1=(\mathfrak{f}_1',\mathfrak{f}_1'')$ is a fundamental family of $(M_1,\mathcal{P})$ representing the Euler structure $\mathfrak{e}_1$.
In the same way we obtain a fundamental family $\mathfrak{f}_2=(\mathfrak{f}_2',\mathfrak{f}_2'')$ 
of $(M_2,\mathcal{P}')$ representing the Euler structure $\mathfrak{e}_2$.

Now, $\mathfrak{f}_1'\cup r_1(j_1(\mathfrak{f}^W))$ is a fundamental basis on $M_1$ and $\mathfrak{f}_2'\cup r_2(j_2(\mathfrak{f}^B))$ is a fundamental basis on $M_2$. Choose on $M$ the fundamental family $\mathfrak{f}=s_1(\mathfrak{f}_1')\cup s_2(\mathfrak{f}_2')\cup t(\mathfrak{f}^G)$ (where $t=s_1\circ r_1\circ j_1\circ i_1:\hat G\rightarrow \hat M$). One easily sees that $\mathfrak{f}$ induces the Euler structure $\mathfrak{e}=\mathfrak{e}_1\cup\mathfrak{e}_2$ and that the chosen fundamental families are compatible with respect to the exact sequences \eqref{gs3},\eqref{gs6},\eqref{gs7}.

Therefore we can apply theorem \ref{milnor}, obtaining seven equalities between torsions. 
The combination of them leads to the result; in the following calculation, all the torsions are computed with respect to the fundamental bases chosen above and the bases of the twisted homologies fixed in Section \ref{Setting}:
$$
\tau^{\varphi}(M;\mathfrak{e},\mathfrak{h})\underset{\eqref{gs3}}{=}(\tau_c)^{-1}\cdot(\tau^{\varphi}(S))^{-1}\cdot\tau^{\varphi}(M_1)\cdot\tau^{\varphi}(M_2)\underset{\eqref{gs2},\eqref{gs6},\eqref{gs7}}{=}
$$
$$
=\tau_b\cdot (\tau_c)^{-1}\cdot\tau_f\cdot\tau_g\cdot\tau^{\varphi}(G)\cdot\tau^{\varphi}(M_1,\overline{W})\cdot\tau^{\varphi}(M_2,\overline{B})\underset{\eqref{gs1}}{=}
$$
$$
=(\tau_a)^{-1}\cdot\tau_b\cdot (\tau_c)^{-1}\cdot\tau_f\cdot\tau_g\cdot(\tau^{\varphi}(Q))^{-1}\cdot\tau^{\varphi}(\overline{V})\cdot\tau^{\varphi}(\overline{C})\cdot\tau^{\varphi}(M_1,\overline{W})\cdot\tau^{\varphi}(M_2,\overline{B})\underset{\eqref{gs4},\eqref{gs5}}{=}
$$
$$
=(\tau_a)^{-1}\cdot\tau_b\cdot(\tau_c)^{-1}\cdot\tau_d\cdot\tau_e\cdot\tau_f\cdot\tau_g\cdot\tau^{\varphi}(M_1,\mathcal{P};\mathfrak{e}_1,\mathfrak{h}_1)\cdot\tau^{\varphi}(M_2,\mathcal{P}';\mathfrak{e}_2,\mathfrak{h}_2)\qedhere
$$
\end{proof}

\begin{nta}\label{rmk.gluing}
Theorem~\ref{thm.gluing} extends easily to the case $\partial M\neq\emptyset$. One have to consider partitions $\mathcal{P},\mathcal{P}_1,\mathcal{P}_2$ of $\partial M_1\cap\partial M_2$, $\partial M\cap\partial M_1$, $\partial M\cap\partial M_2$ respectively; the resulting formula is:
\begin{equation}\notag
\begin{split}
\tau^{\varphi}(M,\mathcal{P}_1\cup &\mathcal{P}_2;\mathfrak{e}_1\cup\mathfrak{e}_2,\mathfrak{h})=\\
&=\mathfrak{T}(\mathfrak{h},\mathfrak{h}_1,\mathfrak{h}_2)\cdot\tau^{\varphi}(M_1,\mathcal{P}\cup\mathcal{P}_1;\mathfrak{e}_1,\mathfrak{h}_1)\cdot\tau^{\varphi}(M_2,\mathcal{P}'\cup\mathcal{P}_2;\mathfrak{e}_2,\mathfrak{h}_2).
\end{split}
\end{equation}
Now the term $\mathfrak{T}(\mathfrak{h},\mathfrak{h}',\mathfrak{h}'')$ contains other factors, coming from exact sequences involving elements of the partitions $\mathcal{P}_1$ and $\mathcal{P}_2$. We omit the details and the proof, that follows the same scheme as above.
\end{nta}

\subsection{Some computations}\label{Some computations}

In what follows we will try to choose the bases of the twisted homologies wisely, in order to simplify the computation of $\mathfrak{T}(\mathfrak{h},\mathfrak{h}_1,\mathfrak{h}_2)$. To this end, we notice that the exact sequences \eqref{gs1},\eqref{gs4},\eqref{gs5} are easy to compute in general, because we know exactly the involved chain complexes:
\begin{enumerate}
\item $Q=Q^+\cup Q^-$ is a finite union of points. Let $Q^-=\{p_1,\dots,p_j\}$, $Q^+=\{p_{j+1},\dots p_k\}$. We have $H_*^{\varphi}(Q)=H_0^{\varphi}(Q)=\oplus_{i=1}^k \mathbb{F}\, p_i$, where we have identified $Q$ with its maximal abelian covering.
\item $G$ is an union of circles. Take one circle $S$; up to subdivision, $S$ is a CW-complex with exactly one vertex $p$ and one edge $e$. If $\varphi(H_1(S))\neq 1$, then $C_*^{\varphi}(S)$ is acyclic (see \cite[Lemma~6.2]{turaev2}), if $\varphi(H_1(S))=1$ then $H_0^{\varphi}(S)\cong H_1^{\varphi}(S)\cong\mathbb{F}$. We define a \emph{canonical basis} on $H_*^{\varphi}(S)$ as the natural bases $\{\hat p,\hat e\}$, where $\hat p$ and $\hat e$ are lifts of $p$ and $e$ such that $\partial\hat e= t\hat p-\hat p$, and $t$ is the generator of the action of $H_1(S)$ on $\hat S$ (see \figurename~\ref{circle}). The canonical basis on $H_*^{\varphi}(G)$ is the union of the canonical bases of all the circles in $G$. 

\begin{figure}[b]

\centering

\includegraphics[scale=0.5]{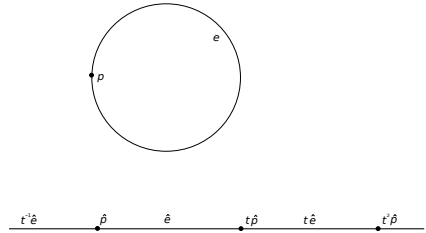}

\caption{A circle and its maximal abelian covering $\mathbb{R}$.}\label{circle}

\end{figure}

\item $\overline{V}$ and $\overline{C}$ are unions of circles and segments. We have already studied the twisted homology of circles in point 2. Notice that segments retracts to points, hence their twisted homology is the same as point 1.

\item Now we study the pair $(\overline{C},Q^+)$ (the same applies to $(\overline{V},Q^-)$) and we fix a canonical basis, as already done for $H_*^{\varphi}(G)$. Each connected component $S$ of $\overline{C}$ has one of the following four forms:

\begin{itemize}
\item $S$ is a circle: we have already studied this case in point 2, and we have already shown how to choose a canonical basis.
\item $S$ is a segment and both points of $\partial S$ belongs to $Q^-$: we have already studied it in point 3. Up to subdivisions, $S$ is a CW-complex with exactly one edge $e$ and two vertices $p_1,p_2\in Q^-$ such that $\partial e=p_2-p_1$. We have $H_0^{\varphi}(S,Q^+)=\mathbb{F}\, p_1$, thus a basis is formed by an element only. As a canonical basis for $H_*^{\varphi}(S)$ we chose $\{[p_1]\}$
\item $S$ is a segment and both points of $\partial S$ belongs to $Q^+$: up to subdivisions, $S$ is a CW-complex with exactly one edge $e$ and two vertices $p_1,p_2$ such that $\partial e=p_2-p_1$. We obtain $H_0^{\varphi}(\overline{C},Q^+)=\{1\}$ and $H_1^{\varphi}(S,Q^+)=\mathbb{F}\, e$. In this case the canonical basis will be $\{[e]\}$.
\item $S$ is a segment and $\partial S$ is formed by a point in $Q^+$ and a point in $Q^-$: 
one easily check that $C_*^{\varphi}(S,Q^+)$ is acyclic. 
\end{itemize}
The union of the canonical bases on the connected components gives the canonical basis on $H_*^{\varphi}(\overline{C},Q^+)$.
\end{enumerate}
These observations allows to easily compute torsions $\tau_a, \tau_d, \tau_e$. We obtain the following:

\begin{lmm}
Let $\{[p_1],\dots,[p_r],[e_1],\dots, [e_s]\}$ be the canonical basis of $(\overline{C},Q^+)$ and $\{[p_{r+1}],\dots,[p_u],[e_{s+1}],\dots, [e_v]\}$ be the canonical basis of $(\overline{V},Q^-)$. Equip $H_*^{\varphi}(G)$ with the canonical basis $\mathfrak{h}^G$. Let $\mathfrak{h}''_1,\mathfrak{h}''_2$ be generic bases of $H_*^{\varphi}(\overline{C},Q^+)$, $H_*^{\varphi}(\overline{V},Q^-)$. 
 Then: 
$$
\{\mathfrak{h}''_1,\mathfrak{h}''_2\}=\{a_1 [p_1],\dots, a_u [p_u],b_1 [e_1],\dots b_v [e_v]\}
$$
for opportune $a_i,b_j\in\mathbb{F}^*$.
With respect to the bases $\mathfrak{h}^G,\mathfrak{h}'_1,\mathfrak{h}'_2$ (regardless of the choice of the bases for the other twisted homologies), we have: 
$$
(\tau_a)^{-1}\cdot\tau_d\cdot\tau_e=\frac{a_1\cdots a_u}{b_1\cdots b_v}.
$$
\end{lmm}

\section{Combinatorial encoding of Euler structures}\label{sctn4}

In Section \ref{stream-spines} and \ref{Combings} we recall the main results of \cite{petronio}: in particular, we define stream-spines and we show that they encode vector fields on a 3-manifold. Using stream-spines, we show how to geometrically invert the reconstruction map $\Psi$ (Theorem~\ref{combinatorial encoding}): this will give us a way to explicitly compute torsions.

\subsection{Stream-spines}\label{stream-spines}

A \emph{stream-spine} $P$ is a connected compact 2-dimensional polyedron such that a neighborhood of each point of $P$ is homeomorphic to one of the five models in \figurename~\ref{spine}.

\begin{figure}[b]

\centering

\includegraphics[scale=0.7]{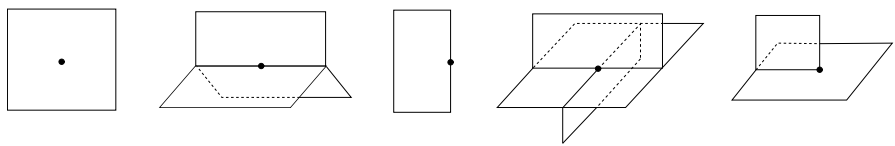}

\caption{Local models of a stream-spine}\label{spine}

\end{figure}

Specifically, a stream-spine $P$ is formed by:
\begin{itemize}
\item some open surfaces, called \emph{regions}, whose closure is compact and contained in $P$;
\item some \emph{triple lines}, to which three regions are locally incident;
\item some \emph{singular lines}, to which only one region is locally incident;
\item some points, called \emph{vertices}, to which six regions are incident;
\item some points, called \emph{spikes}, to which a triple line and a singular line are incident; 
\end{itemize}

A \emph{screw-orientation} on a triple line is an orientation of the line together with a cyclic ordering of the three regions incident on it, viewed up to a simultaneous reversal of both (see \figurename~\ref{branching}-left).

A stream-spine is said to be \emph{oriented} if 
\begin{itemize}
\item each triple line is endowed with a screw-orientation, so that at each vertex the screw-orientations are as in\figurename~\ref{branching}-center;
\item each region is oriented, in such a way that no triple line is induced three times the same orientation by the regions incident to it.
\end{itemize}

\begin{figure}

\centering

\includegraphics[scale=0.6]{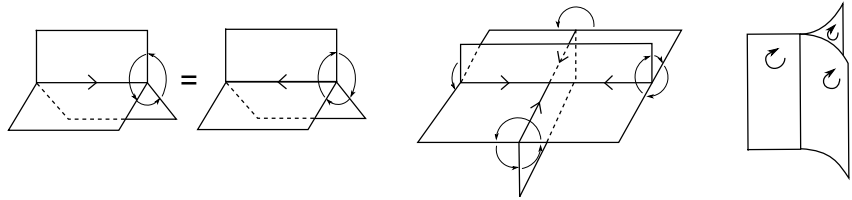}

\caption{Convention on screw orientation, compatibility at vertices and geometric interpretation of branching.}\label{branching}

\end{figure}

Two oriented stream-spines are said to be \emph{isomorphic} if there exists a PL-homomorphism between them preserving the orientations of the regions and the screw-orientations of the triple lines.

We denote by $\mathcal{S}_0$ the set of oriented stream-spines viewed up to isomorphism. An embedding of $P\in\mathcal{S}_0$ into a 3-manifold $M$ is said to be \emph{branched} if every region of $P$ have a well defined tangent plane in every point, and the tangent planes at a singularity $p\in P$ to each region locally incident to $p$ coincide (see \figurename~\ref{branching}-right for the geometric interpretation near a triple line; see \cite[\S~1.4]{petronio} for an accurate definition of branching).

\begin{prp}\label{spines to manifolds}
To each stream-spine $P\in\mathcal{S}_0$ is associated a pair $(\tilde M,\tilde{\mathfrak{v}})$, defined up to oriented diffeomorphism, where $\tilde M$ is a connected 3-manifold  and $\tilde{\mathfrak{v}}$ is a vector field on $\tilde M$ whose orbits intersect $\partial \tilde M$ in both directions.
Moreover, $P$ embeds in a branched fashion in $\tilde M$ and the choice of a cellularization on $P$ induces a cellularization $\tilde{\mathcal{C}}$ on $\tilde M$.
\end{prp}

\begin{proof}
The construction of $\tilde M$ and $\tilde{\mathfrak{v}}$ is carefully analyzed in \cite[Prop.~1.2]{petronio}. One start from the spine, thicken it to a PL-manifold $\hat M$ and then smoothen the angles to obtain a differentiable manifold $\tilde M$. $\tilde{\mathfrak{v}}$ is a vector field everywhere positively transversal to the spine.

It remains to show how to obtain the cellularization $\tilde{\mathcal{C}}$ from the cellularization of $P$. We will do it by thickening the 2-cells of $P$ and then showing how to glue them together along the edges.

Pick a 2-cell $r$  and thicken it to a cylinder $c\cong r\times [-1,1]$. This identification is done in such a way that the original $r$ is identified with $r\times \{0\}$, and the orientation of $r$ (inherited from the branching of $P$) together with the positive orientation on the segment $[-1,1]$ gives the positive orientation of $\mathbb{R}^3$ (see \figurename~\ref{cylinder}). The upper and lower faces $r\times \{1\}$ and $r\times \{-1\}$ will be part of the boundary (so they are not glued with any other quadrilateral); the side surface will be glued with the side surfaces of the other cylinders.

A natural vector field $\mathfrak{v}^c$ is defined on $c$: $\mathfrak{v}^c$ is the constant field whose orbits are rectilinear, directed from $r\times \{-1\}$ to $r\times \{1\}$, and orthogonal to $r\times \{0\}$ (see again \figurename~\ref{cylinder}).

\begin{figure}

\centering

\includegraphics[scale=0.5]{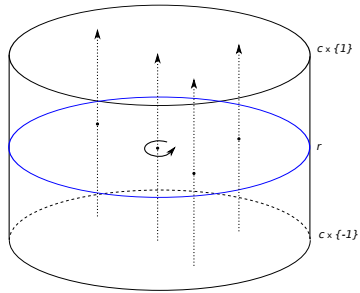}

\caption{Thickening $c$ of the 2-cell $r$ and the vector field $\mathfrak{v}^c$.}\label{cylinder}

\end{figure}

$c$ has a natural cellularization. Let $p_1,\dots, p_k$, $e_1,\dots e_k$ be the vertices and edges composing the boundary of $r$. Then the cells of $c$ are the following:
\begin{enumerate}
\item the vertices are the points $p_i\times \{-1\}$ and $p_i\times \{1\}$, for $i=1,\dots,k$;
\item the edges are the lines $e_i\times \{-1\}$, $e_i\times \{1\}$, $p_i\times [-1,1]$, for $i=1,\dots,k$;
\item the 2-cells are the faces $r\times\{-1\}$, $r\times\{1\}$ and $e_i\times [-1,1]$ for $i=1,\dots,k$;
\item the only 3-cell is $r\times [-1,1]$.
\end{enumerate}

Now we shift our attention from the 2-cells to the edges of the cellularization of $P$. The edges will describe how to modify the side surfaces of the cylinders and how to glue them together.

Pick an edge $e$. Depending on the nature of $e$, we distinguish three cases:
 \begin{itemize}
 \item if $e$ is a regular line (i.e., $e$ is neither a singular nor a triple line), then it is contained in the boundary of two 2-cells $r_1,r_2$. The respective cylinders $c_1,c_2$ are simply glued together along the common face $e\times [-1,1]$.
 \item if $e$ is a singular line, then it is only contained in the boundary of one 2-cell $r$, thus no gluing is needed. We simply collapse the corresponding face $e\times [-1,1]$ to the line $e\times \{0\}$ via the natural projection. Note that this collapse gives rise to a concave tangency line on the boundary (see \figurename~\ref{gluecyl}-center);
 \item if $e$ is a triple line, then there are three 2-cells $r_1,r_2,r_3$ containing the face $e\times [-1,1]$. Recall that $r_1,r_2,r_3$ are oriented (with the orientation inherited from the spine) and that one, say $r_1$, induces on $e$ the opposite orientation with respect to the other two ($r_2,r_3$). Subdivide the cell $e\times [-1,1]$ in $r_1$ into two subcells $e\times [-1,0]$ and $e\times [0,1]$. Glue this two subcells with the corresponding cells on $r_2$ and $r_3$, as shown in \figurename~\ref{gluecyl}-right. Note that this gluing gives rise to a convex tangency line. 
 \end{itemize}
\figurename~\ref{vertex} and \figurename~\ref{spike} show what happens near vertices and spikes.

\begin{figure}

\centering

\includegraphics[scale=0.7]{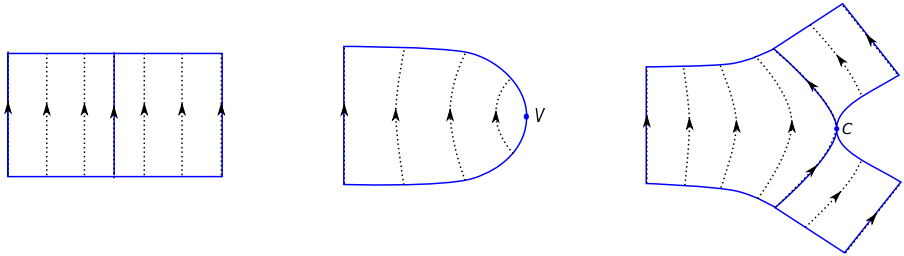}

\caption{Cross-section of gluings and modifications along regular (left), singular (center) and triple (right) line.}\label{gluecyl}

\end{figure}

\begin{figure}

\centering

\includegraphics[scale=0.4]{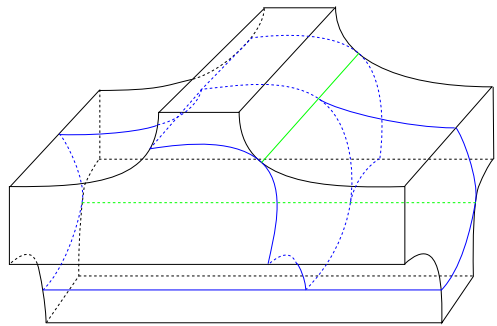}

\caption{Behavior of $\tilde{\mathfrak{v}}$ near a vertex. There are two concave lines corresponding to the two triple lines intersecting in the vertex. Notice that there is exactly one orbit of $\tilde{\mathfrak{v}}$ that is tangent to both the triple lines.}\label{vertex}

\end{figure}

\begin{figure}

\centering

\includegraphics[scale=0.6]{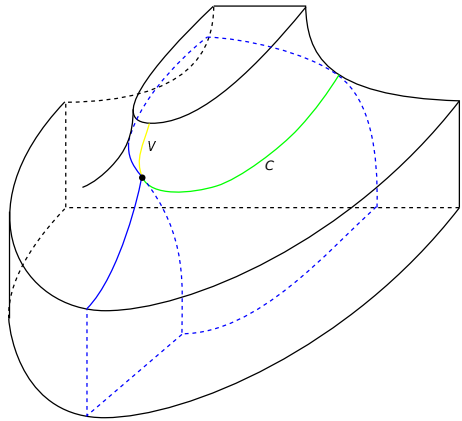}

\caption{Behavior of $\tilde{\mathfrak{v}}$ near a spike. Notice that $\tilde{\mathfrak{v}}$ goes from a concave tangency line (green) to a convex tangency line (yellow), or viceversa, throught a cuspidal point.}\label{spike}

\end{figure}

The gluing of the cylinders $c$, opportunely modified as explained above, and their vector fields $\mathfrak{v}^c$ gives rise to the pair $(\tilde M,\tilde {\mathfrak{v}})$ and to the cellularization $\tilde{\mathcal{C}}$.
\end{proof}

\subsection{Combings}\label{Combings}

The main achievement of \cite{petronio} is to show that stream-spines encode combings, so that they can be used as a combinatorial tool to study vector fields on 3-manifolds. 

Proposition \ref{spines to manifolds} gives us a map $\varphi:\mathcal{S}_0\rightarrow\mathfrak{Comb}$. Unfortunately, this map is not surjective, as the image is formed only by combings $[M,\mathfrak{v}]$ where $\mathfrak{v}$ is a traversing field, i.e., a field whose orbits start and end on $\partial M$. Consider the subset $\mathcal{S}\subset\mathcal{S}_0$ of stream-spines $P$ whose image $\varphi(P)=[\tilde M,\tilde{\mathfrak{v}}]$ contains at least one trivial sphere $S_{triv}$ (i.e., a sphere in $\partial\tilde M$ that is split into one white disc and one black disc by a concave tangency circle). 
Denote by $\Phi(P)$ the combing $[M,\mathfrak{v}]$ obtained from $\varphi(P)$ by gluing to $S_{triv}$ a trivial ball $B_{triv}$ (i.e., a ball endowed with a vector field $\mathfrak{u}$ such that $(\partial B_{triv},\mathfrak{u}|_{\partial B})$ is a trivial sphere) matching the vector fields. This gives a well defined map $\Phi:\mathcal{S}\rightarrow\mathfrak{Comb}$.

\begin{teo}\label{spines-comb}
$\Phi:\mathcal{S}\rightarrow\mathfrak{Comb}$ is surjective.
\end{teo}

\begin{nta}
In \cite{petronio} is also described a set of moves on stream-spines generating the equivalence relation induced by $\Phi$. We will come back to this point in Section \ref{SSS}.
\end{nta}

\begin{nta}
A restatement of the theorem is the following: given a non-singular vector field $\mathfrak{v}$ on a 3-manifold $M$, we can always find a sphere $S\subset M$ that splits $(M,\mathfrak{v})$ into a trivial ball $B$ and a manifold $M\setminus B$ with a traversing field.  
\end{nta}

\subsection{Inverting the reconstruction map}\label{Invertingmap}

Denote by $\mathcal{S}(M,\mathcal{P})\subset\mathcal{S}$ the subset $\Phi^{-1}\!\left(\mathfrak{Comb}(M,\mathcal{P})\right)$. $\Phi$ restricts to a bijection $\mathcal{S}(M,\mathcal{P})\rightarrow\mathfrak{Comb}(M,\mathcal{P})$.
Composing $\Phi$ with the natural projection $\mathfrak{Comb}(M,\mathcal{P})\rightarrow\mathfrak{Eul}^s(M,\mathcal{P})$, we obtain a map $\Xi^s:\mathcal{S}(M,\mathcal{P})\rightarrow\mathfrak{Eul}^s(M,\mathcal{P})$.

We show in this section how to explicitly invert the reconstruction map via stream-spines. To do so, we will exhibit a map $\Xi^c:\mathcal{S}(M,\mathcal{P})\rightarrow\mathfrak{Eul}^c(M,\mathcal{P})$ such that $\Xi^s=\Psi\circ\Xi^c$.

\begin{equation}\label{spines and euler structures}
\begin{tikzcd}[row sep=small, column sep=normal]
& & \mathfrak{Eul}^c(M,\mathcal{P}) \arrow{dd}{\Psi} \\
\mathcal{S}(M,\mathcal{P})\arrow[twoheadrightarrow]{r}{\Phi}\arrow[bend left=15]{urr}{\Xi^c}\arrow[bend right=15]{drr}[swap]{\Xi^s}  & \mathfrak{Comb}(M,\mathcal{P}) \arrow[twoheadrightarrow, shorten >=-3pt, shorten <=-3pt]{dr} & \\
& & \mathfrak{Eul}^s(M,\mathcal{P})
\end{tikzcd}
\end{equation}

Take $P\in\mathcal{S}(M,\mathcal{P})$ and equip it with a cellularization. Recall from Proposition \ref{spines to manifolds} that $P$ induces a combing $\varphi(P)=[\tilde M,\tilde{\mathfrak{v}}]$ and a cellularization $\tilde{\mathcal{C}}$ on $\tilde M$.

Take a point $p_u$ inside each cell $u\in \tilde{\mathcal{C}} \setminus\tilde{\mathcal{C}}_{\partial}$ (where $\tilde{\mathcal{C}}_{\partial}$ is the induced cellularization on $\partial \tilde M$), and denote by $\beta_u$ the arc obtained by integrating $\tilde{\mathfrak{v}}$ in the positive direction, starting from $p_u$, until the boundary is reached. 
Consider the 1-chain:
$$
\tilde{\xi}(P)=\sum_{u\in \mathcal{C}} (-1)^{\dim u}\cdot\beta_u.
$$
Recall that $\Phi(P)=[M,\mathfrak{v}]$ is obtained from $[\tilde M,\tilde{\mathfrak{v}}]$ by gluing a trivial ball on a trivial sphere $S_{triv}$ in $\partial\tilde M$. Thus we have a projection $\pi:\tilde M\rightarrow M$, obtained by collapsing $S_{triv}$ to a point $x_0$, and a cellularization $\mathcal{C}=\pi(\tilde{\mathcal{C}})$ of $M$. It is easily seen that $\mathcal{C}$ is suited to the partition $\mathcal{P}$. Now consider the 1-chain $\xi(P)=\pi(\tilde{\xi}(P))$.

\begin{lmm}
$\xi(P)$ is a combinatorial Euler chain, and the class $[\xi(P)]\in\mathfrak{Eul}^c(M,\mathcal{P})$ does not depend on the cellularization chosen on $P$.
\end{lmm}

\begin{proof}
We first prove that $\xi(P)$ is an Euler chain.
It is easily seen that $\partial\xi(P)$ contains, with the right sign, a point in each (open) cell of $\tilde M$, except for the cells of $W\cup V\cup Q^+$, as wished. It remains to prove that the resulting chain $\partial\xi(P)$ contains the singularity $x_0$ with coefficient 1. This coefficient is the sum of the coefficients of the cells in $B\cap S_{triv}$, and the conclusion follows from $\chi(B\cap S_{triv})=\chi(\mathrm{open\ disk})=1$. 

The fact that $[\xi(P)]$ does not depend on the cellularization of $P$ follows from the next theorem.
\end{proof}

\begin{figure}

\centering

\includegraphics[scale=0.7]{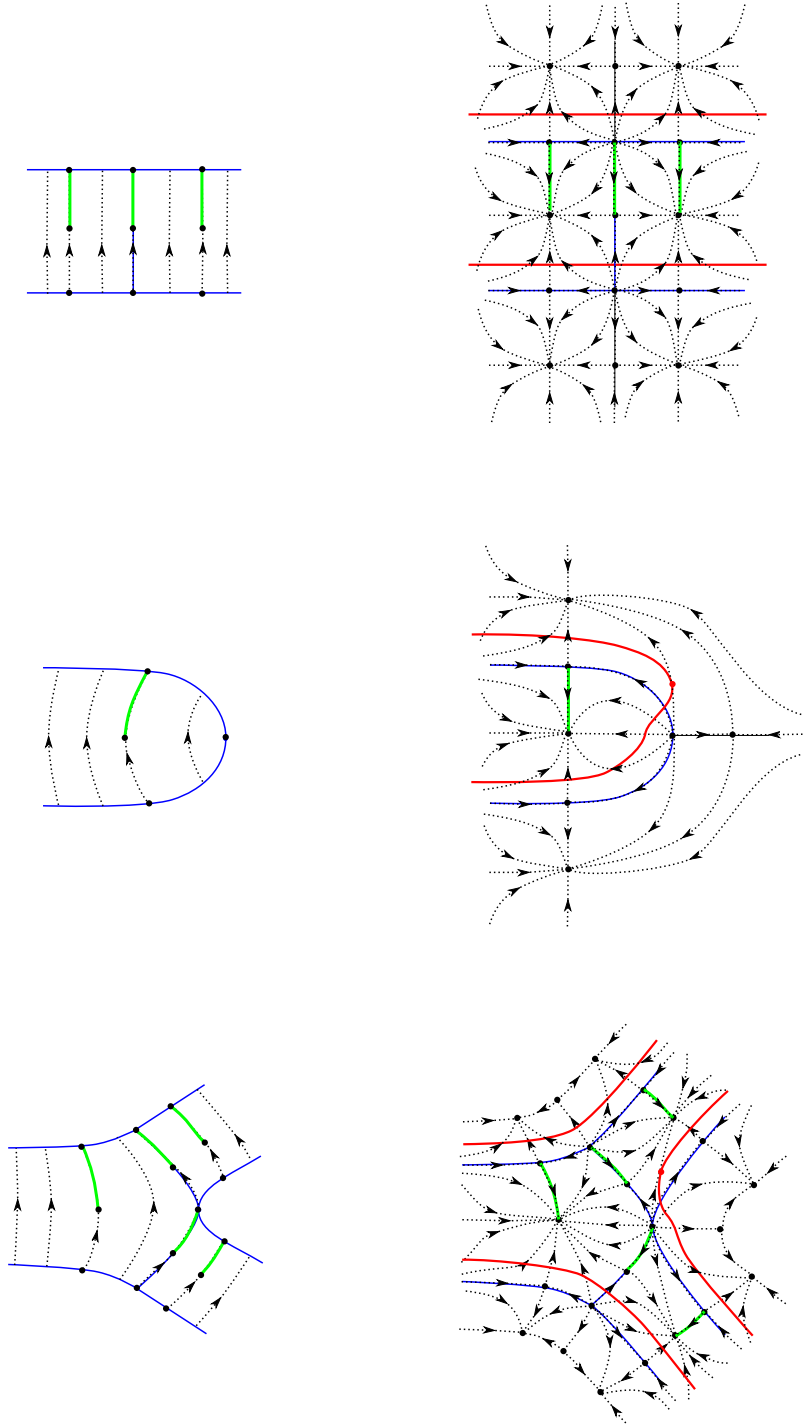}

\caption{Comparison between the field $\mathfrak{v}$ (left) and $\mathfrak{w}_{\tilde{\mathcal{C}}}$ (right) along regular, singular and triple lines. Notice that $\mathfrak{v}$ and $\mathfrak{w}_{\tilde{\mathcal{C}}}$ are antipodal only on $S$ (in green).}\label{desingularization}

\end{figure}

\begin{teo}\label{combinatorial encoding}
$\Psi([\xi(P)])=\Xi^s(P)$. Thus the map that completes diagram \eqref{spines and euler structures} is defined by $\Xi^c(P)=[\xi(P)]$.
\end{teo}

\begin{proof}
Let $\mathfrak{w}_{\mathcal{C}}$ be the fundamental field of the cellularization $\mathcal{C}$. Recall from Theorem \ref{reconstruction map} that the representative of $\Psi(\xi(P))$ is obtained by identifying $M$ with a collared copy $M_h$ of itself (the boundary of $M_h$ is shown in red in \figurename~\ref{desingularization}), then applying a desingularization procedure to $\mathfrak{w}_{\mathcal{C}}$ in a neighborhood of $\xi(P)$. 
It should be noted that our cellularization $\mathcal{C}$ does not satisfy (Hp3) (in fact, the star at each spike differs from the one pictured in \figurename~\ref{functiong}-left), thus the construction of $h$ in Theorem \ref{reconstruction map} does not apply directly. However, it is clear that a suitable function $h$ can be defined (recall Remark~\ref{fund field}): the behavior of $\partial M_h$ near regular, singular and triple line  is shown in red in \figurename~\ref{desingularization}-right; the construction of $h$ near spikes is a bit more complicated, but still analogous to the construction of $h$ near cuspidal points in the proof of Theorem \ref{reconstruction map}.

It is easily seen that every connected component of the support $S$ of $\xi(P)$ is contractible; therefore two different desingularizations of $\mathfrak{w}_{\mathcal{C}}$ represent the same Euler structure. Thus, it is enough to prove that $\mathfrak{v}$ is homologous to any desingularization of $\mathfrak{w}_{\mathcal{C}}$. In particular, it is enough to exhibit a desingularization that is everywhere antipodal to $\mathfrak{v}$.

We will do it in two steps:
\begin{itemize}
\item We prove that the set of points where $\mathfrak{w}_{\mathcal{C}}$ is antipodal to $\mathfrak{v}$ is contained in $S$;
\item We provide a desingularization of $\mathfrak{w}_{\mathcal{C}}$ in a neighborhood of $S$ to a field that is nowhere antipodal to $\mathfrak{v}$ in the neighborhood.
\end{itemize}

We will prove the two claims working with $\tilde M$ 
(proving the formula on $\tilde M$ easily implies the formula on $M$).
Notice that the cells of $\tilde{\mathcal{C}}$ are union of orbits of both $\mathfrak{w}_{\tilde{\mathcal{C}}}$ and $\mathfrak{v}$, hence we can analyze cells separately. Consider one of the cylinders $c$ of the cellularization $\tilde{\mathcal{C}}$. \figurename~\ref{desingularization} shows a cross-section of $c$ and of the vector fields $\mathfrak{w}_{\tilde{\mathcal{C}}}$ and $\mathfrak{v}$: we see that they are antipodal only in $S$ and it is easy to construct the wished desingularization.
\end{proof}

\begin{nta}
In \cite[\S~3]{petronio} is described how to explicitly invert the map $\Phi$. Therefore Theorem \ref{combinatorial encoding} is an effective way to invert the reconstruction map: in details, one starts from a representative $\mathfrak{v}$ of a smooth Euler structure $\mathfrak{e}$, constructs the spine $P=\Phi^{-1}(\mathfrak{v})$ and applies $\Xi^c$ to $P$.
\end{nta}

\subsection{Standard stream-spines}\label{SSS}

We consider for a moment a standard spine $P$, i.e., a spine whose local models are the first, second and fourth of \figurename~\ref{spine} only. This is the spine used in \cite[\S~3]{benedetti} to invert the reconstruction map for Euler structures relative to partitions without cuspidal points. It is easy to prove that one can transform each region of $P$ in a 2-cell using sliding moves; hence the stratification of singularities gives a cellularization of $P$. 

The same approach does not work with a stream-spine $P$, and we are left without a way to obtain a natural cellularization of $P$.
In this section we show how to solve this problem by enriching the structure of a stream-spine with two new local models and a new sliding move.

A \emph{standard stream-spine} $P$ is a connected 2-polyedron whose local models are the five in \figurename~\ref{spine}, plus the two in \figurename~\ref{supersingularity}; specifically, in addition to regular points, triple lines, singular lines, vertices, spikes, we allow:
\begin{enumerate}
\item some \emph{bending lines} (\figurename~\ref{supersingularity}-left), i.e., lines which are induced the same orientation by the two regions incident on it;
\item some \emph{bending spikes} (\figurename~\ref{supersingularity}-right), i.e., points where a singular, a triple and a bending line meet.
\end{enumerate}
Moreover, we require the components of the stratification of singularities to be open cells. 
Denote by $\mathfrak{S}_0$ the set of standard stream-spines. 

\begin{figure}

\centering

\includegraphics[scale=0.7]{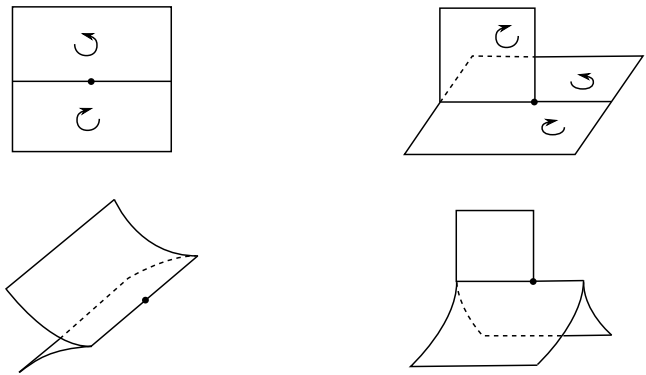}

\caption{New local models and their geometric interpretation.}\label{supersingularity}

\end{figure}

\begin{figure}

\centering

\includegraphics[scale=0.7]{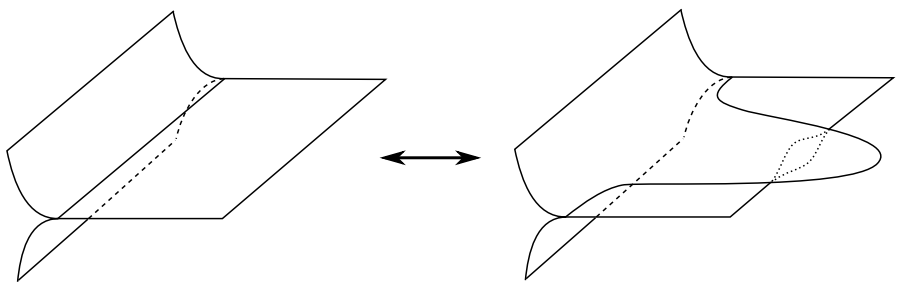}

\caption{The new sliding move consists in digging the triple line until the singular line is crossed.}\label{slidingmove}

\end{figure}

In addition to those described in \cite[\S~2.2]{petronio}, we define a new sliding move on $\mathfrak{S}_0$ as the one depicted in \figurename~\ref{slidingmove}. Obviously, each standard stream-spine can be transformed into a stream-spine by applying the reversal of our sliding move to each bending line. This gives a natural map $\psi:\mathfrak{S}_0\rightarrow\mathcal{S}_0$. 

Consider now the set $\mathfrak{S}$ of standard stream-spines whose image is a stream-spine in $\mathcal{S}$. 

\begin{lmm}
The restriction $\psi:\mathfrak{S}\rightarrow\mathcal{S}$ is surjective.
\end{lmm}

\begin{proof}
It is enough to prove that each region of a stream-spine $P\in\mathcal{S}$ can be divided into a certain number of 2-cells by means of sliding moves. By definition, $P$ contains a trivial sphere $S$, i.e., a sphere formed by two disks glued together along a triple line $t$, such that (1) the two disks induce the same orientation on $t$, and (2) $P$ does not intersect the inner part of $S$.
Consider a region $r$ of $P$.  If $r$ contains no closed singular lines, the old sliding moves are enough to split $r$ into 2-cells. If $r$ contains a closed singular line $s$, we can slide $t$ over other triple lines until we reach $s$ (this can be done by means of the old sliding moves), then use our new sliding move to split $s$ into a singular and a bending line.
\end{proof}


Now we can repeat the arguments of Section \ref{Invertingmap} working with a spine $P\in\mathfrak{S}$ and the surjection $\Phi\circ\psi:\mathfrak{S}\rightarrow\mathfrak{Comb}$. The advantage is that now $P$ is already endowed with a natural cellularization and we do not need to choose one. 
 
It is easy to see how the thickening in the proof of Proposition~\ref{spines to manifolds} works near the new local models. On standard stream-spines we can even describe a different cellularization of $\tilde M$, more in the spirit of \cite{benedetti}, by associating a simplex to each singularity: 
\begin{itemize}
\item to each vertex we associate a truncated tetrahedron (\figurename~\ref{newthick}-left), i.e., a simplex whose faces are four hexagons and four triangles;
\item to each spike and to each bending spike we associate a tetrahedron with a different truncation (\figurename~\ref{newthick}-right): his faces are one hexagon, four quadrilaterals and two triangles.
\end{itemize}
The simplices are then glued together as dictated by the spine (the ideas are the same as \cite[Thm.~1.1.26]{matveev}).
The results of Section \ref{Invertingmap} can be recovered, without significant modifications, working with either the old or the new cellularization.

\begin{figure}

\centering

\includegraphics[scale=0.9]{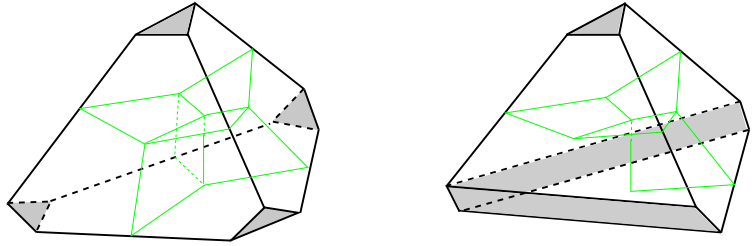}

\caption{Thickening of the singularities. The green part represents the immersion of the spine inside the manifold. The grey faces will form the boundary of the manifold; the white faces are glued with the white faces of other simplices.}\label{newthick}

\end{figure}

\end{document}